\documentclass[12pt]{article}

\usepackage{amsmath}
\usepackage{amsthm}
\usepackage{amsfonts}
\usepackage{mathrsfs}
\usepackage{stmaryrd}
\usepackage{setspace}
\usepackage{fullpage}
\usepackage{amssymb}
\usepackage{breqn}
\usepackage{enumitem}
\usepackage{bbold}
\usepackage{authblk}
\usepackage{comment}
\usepackage{hyperref}
\usepackage{pgf,tikz}
\usepackage{graphicx}
\usepackage{subcaption}

\newtheorem{thm}{Theorem}[section]

\newtheorem{lemma}[thm]{Lemma}

\newtheorem{prop}[thm]{Proposition}
\newtheorem{claim}[thm]{Claim}

\newtheorem{clm}[thm]{Claim}

\usepackage{subcaption} 

\newcommand\ex{\ensuremath{\mathrm{ex}}}

\newcommand\cE{{\mathcal E}}
\newcommand\cF{{\mathcal F}}
\newcommand\cG{{\mathcal G}}
\newcommand\cH{{\mathcal H}}

\newcommand\cK{{\mathcal K}}

\newcommand\cM{{\mathcal M}}

\def\lf{\left\lfloor}
\def\rf{\right\rfloor}

\newtheorem*{thm*}{Theorem}
\newtheorem*{prop*}{Proposition}
\newcommand{\ignore}[1]{}

\title{On the Turán number of blow-ups of $\cF_5$}
\author{
Xiamiao Zhao\thanks{\small Department of Mathematical Sciences, Tsinghua University, Beijing 100084, China. Email:
\small zxm23@mails.tsinghua.edu.cn}\,,
\hspace{0.2em}
Xin Cheng\thanks{\small School of Mathematics and Statistics, Northwestern Polytechnical University and Xi'an-Budapest Joint Research Center for Combinatorics, Xi'an 710129, Shaanxi, P.R. China. Email:
\small \texttt{xincheng@mail.nwpu.edu.cn}.}\,, \hspace{0.2em}
D\'{a}niel Gerbner\thanks{\small Alfr\'ed R\'enyi Institute of Mathematics. Email:
\small \texttt{gerbner.daniel@renyi.hu}.}\,, \hspace{0.2em} 
Hilal Hama Karim$^\ddagger$\thanks{\small Department of Computer Science and Information Theory, Faculty of Electrical Engineering and Informatics, Budapest University of Technology and Economics, Műegyetem rkp. 3., H-1111 Budapest, Hungary. E-mail: \texttt{hilal.hamakarim@edu.bme.hu}.}\,, \hspace{0.2em}

Shujing Miao\thanks{\small School of Mathematics and Statistics, and Hubei Key Lab--Math. Sci., Central China Normal University, Wuhan 430079, China. Email:
\small \texttt{sjmiao2020@sina.com}.}\,, \hspace{0.2em}
Yichen Wang\thanks{\small \textit{Corresponding author.} Department of Mathematical Sciences, Tsinghua University, Beijing 100084, China. Email:
\small wangyich22@mails.tsinghua.edu.cn}\,,
\hspace{0.2em}
Junpeng Zhou\thanks{\small  Department of Mathematics, Shanghai University, Shanghai 200444, P.R. China. Email:
\small \texttt{junpengzhou@shu.edu.cn}.} \thanks{\small Newtouch Center for Mathematics of Shanghai University, Shanghai 200444, P.R. China.}}

\date{}

\begin{document}

\maketitle

\begin{abstract} 
Let $\cF_5$ denote the $3$-uniform hypergraph on the vertex set $\{f_1,f_2,\dots,f_5\}$ with hyperedges $\{f_1f_2f_3,f_1f_2f_4,f_3f_4f_5\}$.
Recently, Balogh, Clemen and Luo determined the Tur\'an number of a one-vertex blow-up of $\cF_5$, more specifically, they blow up the vertex $f_5$ to $t$ vertices, the resulting hypergraph is denoted by $\cF_5(f_5;t)$. They show that for infinitely many $t$, $\cF_5(f_5;t)$ has exponentially many extremal constructions and positive Tur\'an density.

In this paper, we determine the exact Tur\'an number of the hypergraph obtained by blowing up $f_3$ of $\cF_5$ to $t$ vertices and show that it also has exponentially many extremal constructions.

We also give a general upper bound and lower bound of the Tur\'an number of every blow-up of $\cF_5$.
For some special blow-ups of $\cF_5$, for example, $t$-disjoint copies of $\cF_5$, we determine the exact Tur\'an number.

We construct a hypergraph $\cF_{sim}(t)$ which is a subgraph of a blow-up of $\cF_5$, and is contained in the hypergraph obtained by adding any new hyperedge to the Tur\'an hypergraph (the balanced complete $3$-partite hypergraph), but its extremal construction is not the Tur\'an hypergraph.
We also determine the exact Tur\'an number of $\cF_{sim}(t)$.

\end{abstract}



\section{Introduction}

For a $r$-graph $\cG$, we use $e(\cG)$ or $|\cG|$ to denote the number of hyperedges contained in $\cG$.
Let $\mathcal{F}$ be an $r$-uniform hypergraph ($r$-graph for short). A hypergraph $\cH$ is \textit{$\mathcal{F}$-free} if $\cH$ does not contain $\mathcal{F}$ as a subhypergraph. The \textit{Tur\'{a}n number} of $\mathcal{F}$ is the maximum number of hyperedges in an $n$-vertex $\mathcal{F}$-free $r$-graph, which is denoted by ${\rm{ex}}_r(n,\mathcal{F})$. The \textit{Tur\'an density} of $\cF$ is defined as $\pi(\cF):=\lim_{n\to\infty}\ex_r(n,\cF)/\binom{n}{r}$, and the existence of the limit follows from a simple averaging argument of Katona, Nemetz, and Simonovits \cite{KaNS}. Let $T_r(n,k)$ denote the $n$-vertex $r$-uniform Tur\'an hypergraph with $k$ parts, i.e., the complete $k$-partite $r$-graph on $n$ vertices with each part of order $\lfloor n/k\rfloor$ or $\lceil n/k\rceil$. 

A classical result in extremal graph theory is Tur\'{a}n's theorem \cite{Tu}, which determines the Tur\'{a}n number of complete graphs. The Erd\H{o}s-Stone-Simonovits theorem \cite{ESt,ESi} determines Tur\'{a}n density for $r$-chromatic graphs $F$. More precisely, $\text{ex}(n,F)=\big(\frac{r-2}{r-1}\big) \frac{n^2}{2}+o(n^2)$.
In particular, this gives the asymptotics of the Tur\'an number of every graph with chromatic number at least three. 

A \textit{blow-up} of a hypergraph $\cH$ is obtained by replacing each vertex $v_i$ by $t_i$ copies $v_i^1,\dots, v_i^{t_i}$, and each hyperedge $v_1v_2\dots v_r$ by the corresponding complete $r$-partite hypergraph, i.e., by the hyperedges $v_1^{a_1}v_2^{a_2}\dots v_r^{a_r}$, where $1\le a_i\le t_i$ for each $i\le r$. A well-known result is that a hypergraph and its blow-ups have the same Tur\'an density, see e.g. \cite{keevash} for a proof. Note that this statement and Tur\'an's theorem imply the Erd\H{o}s-Stone-Simonovits theorem. 

In this paper, we focus on a classical $3$-uniform hypergraph.
Let $\cF_5$ be the 3-uniform hypergraph with vertex set $\{f_1,f_2,f_3,f_4,f_5\}$ and hyperedge set $\{f_1f_2f_3,f_1f_2f_4,f_3f_4f_5\}$~(see Figure~\ref{fig: F5}).
Frankl and F\"uredi \cite{ff} showed that for $n\ge 3000$, $\ex_3(n,\cF_5)=|T_3(n,3)|$. Keevash and Mubayi \cite{kemu} improved the threshold 
to $n\ge 33$. 

Usually, when the Tur\'an density is zero for a hypergraph, we say it is the \textit{degenerate case}.
Otherwise, we say it is the \textit{non-degenerate case}.
Recently, Balogh, Clemen, and Luo \cite{bcl} found a hypergraph with exponentially many non-isomorphic extremal hypergraphs in the non-degenerate case. In the degenerate case, hypergraphs with that property were known, for example, some designs. 
An \textit{$(n,k,r,t)$-design} is an $n$-vertex $k$-uniform hypergraph with the property that every $r$-set is contained in exactly $t$ hyperedges. Clearly, these are the extremal hypergraphs for forbidding $t+1$ hyperedges sharing the same set of $r$ vertices. Keevash~\cite{Ke2} proved that such designs exist if $n$ is sufficiently large and some divisibility conditions hold.
In another paper~\cite{kee} he showed that under the above conditions, there are exponentially many non-isomorphic $(n,k,r,t)$-designs.

Let $\cF_5(f_i;t)$ be obtained from $\cF_5$ by blowing up the vertex $f_i$ to $t$ vertices~(see Figure~\ref{fig:F5(f3;t)}). For example, $\cF_5(f_3;2)=\{f_1f_2f_3,f_1f_2f_4,f_3f_4f_5,f_1f_2f_6,f_6f_4f_5\}$, where $f_6$ is a blow up vertex of $f_3$.

\begin{figure}[t]
    \centering
    \begin{subfigure}[b]{0.35\textwidth}
        \centering
        \includegraphics[width=\textwidth]{ 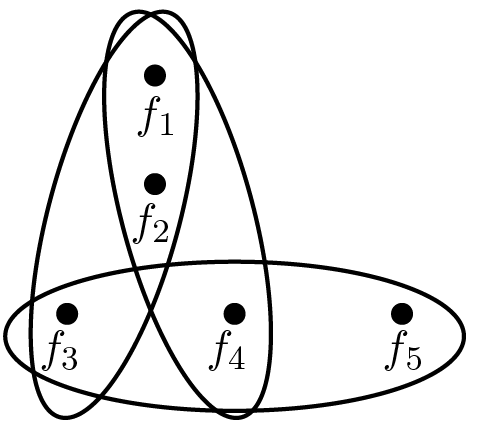}
        \caption{The hypergraph $\cF_5$.}
        \label{fig: F5}
    \end{subfigure}
	\hspace{0.05\textwidth}
    \begin{subfigure}[b]{0.35\textwidth}
    \centering
       \includegraphics[width=1.1\linewidth]{ 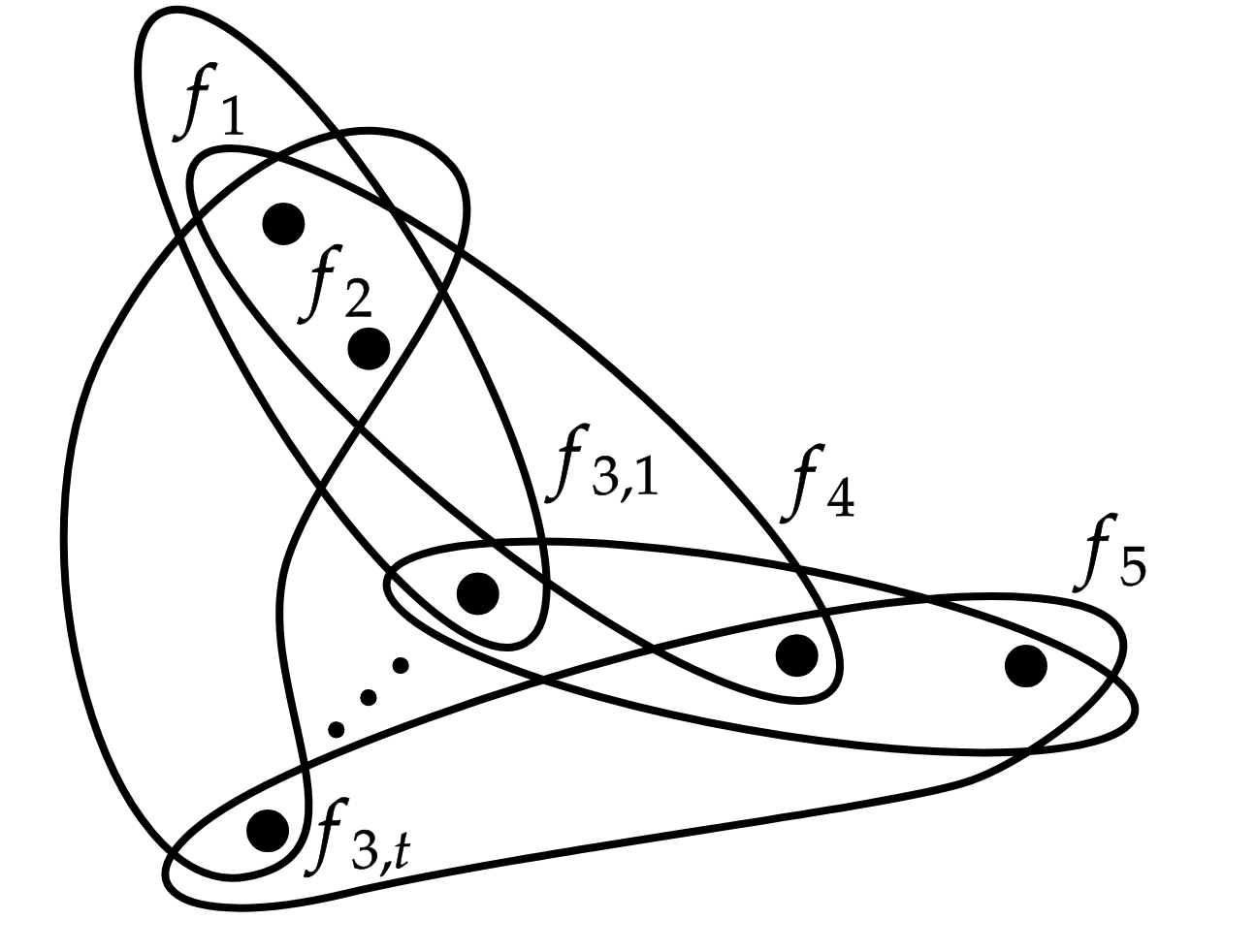}
    \caption{The graph $\cF_5(f_3;t)$.}
    \label{fig:F5(f3;t)}
    \end{subfigure}
    \caption{}
\end{figure}

Let $D(n,3,t,2)$ denote a maximum $n$-vertex $3$-uniform hypergraph such that every pair of vertices is contained in at most $t-1$ hyperedges.
Balogh, Clemen and Luo \cite{bcl} determined the extremal construction of $\cF_5(f_5;t)$, which is obtained by embedding $D(\left\lfloor n/3\right\rfloor,3,t,2)$ or $D(\left\lceil n/3\right\rceil,3,t,2)$ into every part of $T_3(n,3)$. 
We denote it as $T_3(n,3)^{t+}$.
This shows the hypergraph $\cF_5(f_5;t)$ has exponentially many non-isomorphic extremal constructions.

\begin{thm}[\cite{bcl}] \label{thm: F5,f1,t}
For sufficiently large $n$, we have $$\ex_3(n,\cF_5(f_5;t))=e(T_3(n,3)^{t+}).$$
\end{thm}
Any blow-up of $
\cF_5$ has the same Tur\'an density as $\cF_5$, but the exact value is unknown. Following the research of \cite{bcl}, we initiate the study of what happens if we blow up other vertices of $\cF_5$.


First, we give a general result for all the blow-ups of $\cF_5$.
Let $\cF_5^S(m)$ denote the hypergraph obtained by adding one vertex $v$ to $T_3(3m,3)$ and one edge $E$ containing $v$ and two vertices in one part of the $T_3(3m,3)$. Not that $\cF_5$ is a subgraph of $\cF_5^S(2)$.
It is shown in \cite{bcl} that for all $m\geq 2$,  $\ex_3(n,\cF_5^S(m))=e(T_3(n,3))$.

For integers $t+s \geq 3$ and $m\geq 2s$,
let $\cF_5^S(m,t,K_{s,s})$ denote the hypergraph obtained in the following way. We take a $T_3(3m,3)$ and in one of its parts we take two disjoint $s$-sets $S_0$ and $S_1$. We take a set $T$ of $t$ additional vertices. We take the hyperedges of $T_3(3m,3)$, and every hyperedge that contains one vertex from $T$, one vertex from $S_0$, and one vertex from $S_1$~(see Figure \ref{fig: P3}).
In other words, the link graph of every vertex in $T$ is a $K_{s,s}$ with parts $S_0$ and $S_1$ (the \textit{link hypergraph} of the vertex $v$ of an $r$-graph $\cH$ is an $(r-1)$-graph, defined as $L(\cH(v))=\big\{e\backslash\{v\}\,:\,e\in E(\cH), v\in e\big\}$). 

\begin{figure}[t]
    \centering
    \begin{subfigure}[b]{0.35\textwidth}
        \centering
\includegraphics[width=1.0\linewidth]{ 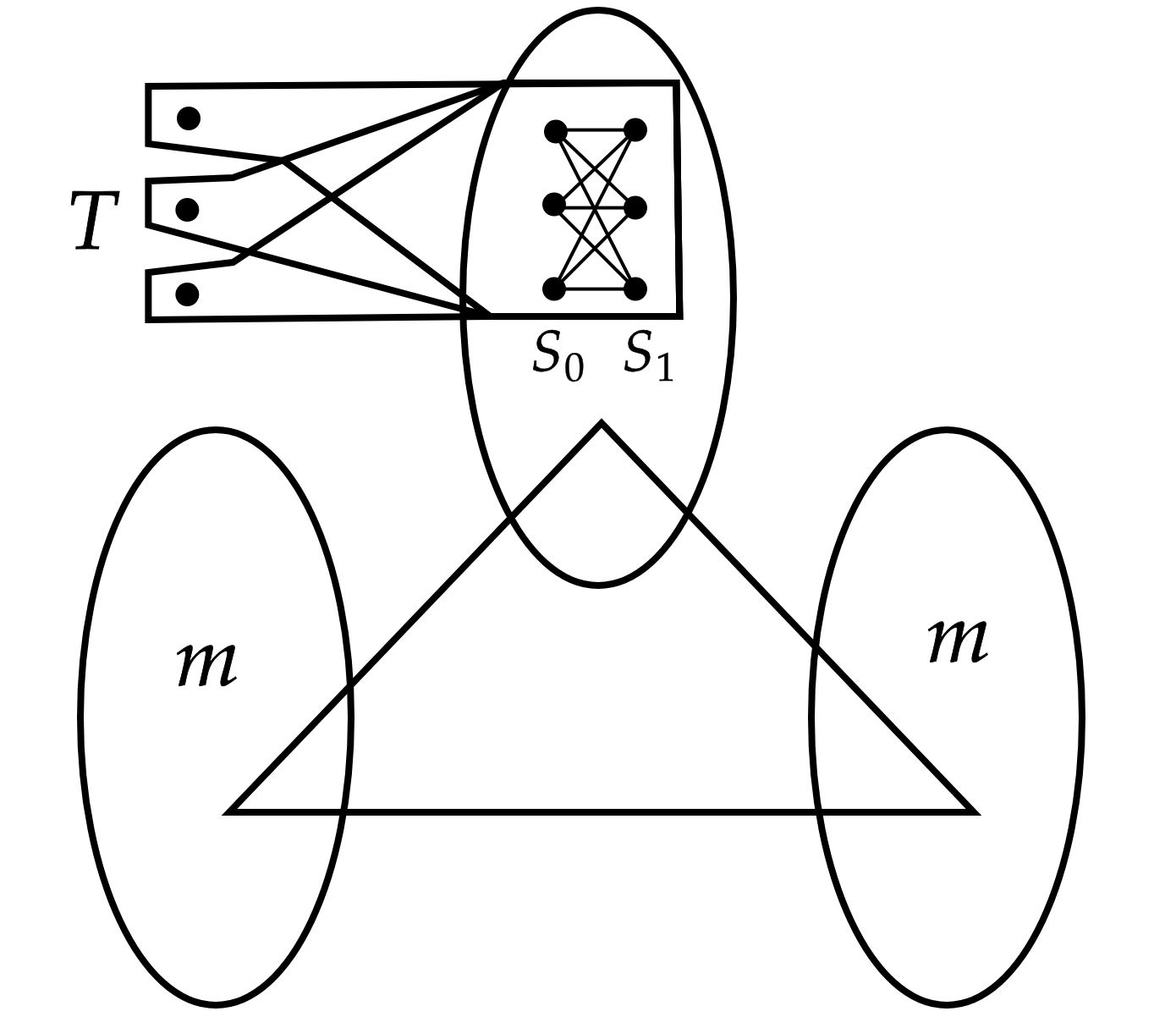}
        \caption{The hypergraph $\cF_5^S(m,t,K_{s,s})$}\label{fig: P3}
    \end{subfigure}
    \hspace{0.01\textwidth}
    \begin{subfigure}[b]{0.35\textwidth}
\includegraphics[width=1.0\textwidth]{ 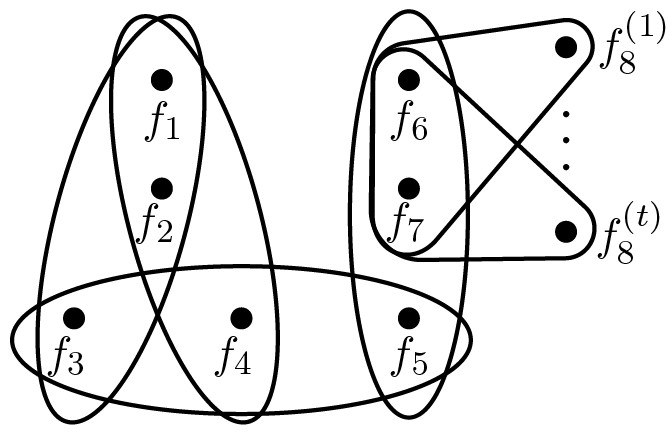}
        \caption{The hypergraph $\cF_{sim}(t)$}\label{fig: F5 sim}
    \end{subfigure}
    \caption{}
\end{figure}


Note that $\cF_5^S(m)=\cF_5^S(m,1,K_{1,1})$.
Also note that every  $t$-blow-up of $\cF_5$ (i.e. we blow up each vertex $t$ times) is a subgraph of $\cF_5^S(m,t,K_{s,s})$ for some $m,t$ and $s\geq t$. 
Indeed, we can embed the blow-ups of $ f_1$ and $ f_2$ into the two parts of $T_3(3m,3)$ that do not contain $S_0$ or $S_1$, and embed the blow-ups of $f_3$ into $S_0$, the blow-ups of $f_4$ into $S_1$, and the blow-ups of $f_5$ into $T$.

To state our general theorem, we need to introduce a certain ``bipartite'' hypergraph Tur\'an problem. We will consider $3$-uniform hypergraphs with $m+n$ vertices and a bipartition of the vertex set into $A,B$ with $|A|=m$ and $|B|=n$, whose hyperedges consist of one vertex from $A$ and two vertices from $B$.
We forbid copies of $K_{t,s,s}$ in which the $t$-vertex part is in $A$ and the rest of the vertices are in $B$.
Here $K_{t,s,s}$ is a complete $3$-uniform $3$-partite hypergraph where the parts have order $t,s$ and $s$, respectively. 
In other words, if we consider the link graphs of the vertices of $A$, no copy of $K_{s,s}$ inside $B$ is contained in $t$ many of those link graphs.
Let $\ex_3^{bip}(m,n,K_{t,s,s})$ denote the maximum number of hyperedges in a $3$-uniform hypergraph under the above restriction.
In the following theorem, we give a lower bound and an upper bound of the Tur\'an number of $\cF_5^S(m,t,K_{s,s})$.

\begin{thm}\label{thm: cF_5^S(m,B)} 
   When $n$ is large enough, for positive integers $m,t$ and $s$, with $m\geq 2s\geq 4$, we have
    $$\ex_3(n,\cF_5^S(m,t,K_{s,s}))\leq e(T_3(n,3))+6\cdot\ex^{bip}_3(n/3,n/3,K_{t,s,s})+3\cdot\ex_3(n/3,K_{t,s,s})+o(n^{3-\frac{1}{s^2}}),$$
    and 
    $$\ex_3(n,\cF_5^S(m,t,K_{s,s}))\geq e(T_3(n,3))+3\cdot\ex_3(n/3,K_{t,s,s}).$$
    Moreover, when $t=1$,
    $$\ex_3(n,\cF_5^S(m,1,K_{s,s}))\leq e(T_3(n,3))+6\cdot\ex^{bip}_3(n/3,n/3,K_{1,s,s})+3\cdot\ex_3(n/3,K_{1,s,s})+o(n^{3-\frac{1}{s}}).$$
\end{thm}
Here, note that the terms $\ex_3^{bip}(n/3,n/3,K_{t,s,s})$ and $\ex_3(n/3,K_{t,s,s})$ are both $O(n^{3-\frac{1}{s^2}})$ according to Theorems \ref{thm: ex3 of Ktss} and \ref{thm: ex bip Ktss} which we prove in Section 3. The order of magnitude of them is not determined for $t\geq 2$, but when $t=1$, we have a sharper bound $O(n^{3-\frac{1}{s}})$. 
Thus, for the upper bound, the first term is the leading term, the second and third terms are the second-leading terms, and then the term $o(n^{3-\frac{1}{s^2}})$ ($o(n^{3-\frac{1}{s}})$ when $t=1$) may be negligible.

We also provide some exact results for certain subgraphs of the blow-up of $\cF_5$. In \cite{bcl}, Balogh, Clemen and Luo asked the following: among all subgraphs of a blow-up of $\cF_5$, for which hypergraph $F$, $\ex_3(n,F)=|T_3(n,3)|$ holds for sufficiently large $n$? According to the graph case, one natural guess is that $\ex_3(n,\cF)=|T_3(n,3)|$ if and only if $\cF$ is a subgraph of $T_3(m,3)+E$~(i.e. adding another hyperedge to $T_3(m,3)$) for some integer $m$. This condition is clearly necessary, and Simonovits \cite{S} showed that in the graph case, the analogous condition is also sufficient. However,  Balogh, Clemen, and Luo \cite{bcl} gave a complicated counterexample, showing that this condition is not sufficient.

Here, we present a simpler counterexample. Let $\cF_{sim}(t)$ denote the $3$-uniform hypergraph obtained by blowing up the vertex $f_8$ $t$ times in the hypergraph $\{f_1f_2f_3,f_1f_2f_4,f_3f_4f_5,f_5f_6f_7,\\f_6f_7f_8\}$ (see Figure \ref{fig: F5 sim}). 
Note that $\cF_{sim}(t)$ is contained in a blow-up of $\cF_5$,
 and it is easy to check for sufficiently large  $n$, 
 adding any other hyperedges to $T_3(n,3)$ will create a copy of $\cF_{sim}(t)$. But we show that $\cF_{sim}(t)$ has a very different extremal structure.

\begin{figure}[t]
    \centering
\centering
\includegraphics[width=0.35\textwidth]{ 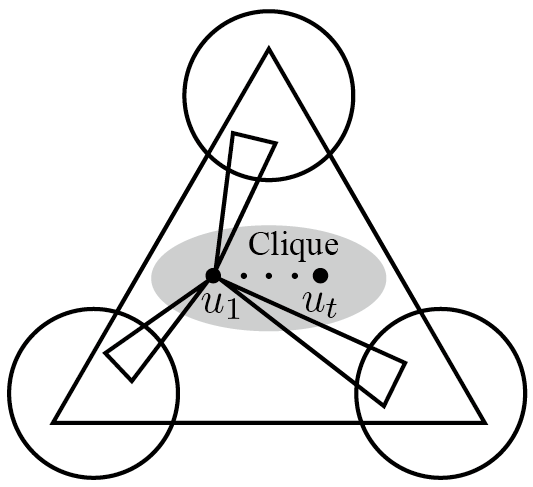}
		\caption{The unique extremal graph of $\cF$}
        \label{fig: unique extremal of F}

\end{figure}

We provide the exact Tur\'an number of $\cF_{sim}(t)$ in the following theorem.
\begin{thm}\label{thm: F5 v8 t}
    When $n$ is sufficiently large, we have
    \[
        \ex_3(n, \cF_{sim}(t)) = \max_{a+b+c =n-t} \left\{ abc + t \left( \binom{a}{2} + \binom{b}{2} + \binom{c}{2} \right) \right\} + \binom{t}{3},
    \]
    where $a,b,c$ are positive integers.
    The maximum is achieved when $a,b,c$ differ by at most $1$.
    Equality holds if and only if the extremal hypergraph is obtained from the complete tripartite hypergraph on $n-t$ vertices $T_3(n-t,3)$ and adding $t$ vertices forming a clique, and each vertex of this clique with each pair that lies in one part of the $T_3(n-t,3)$ forms a hyperedge~(see Figure~\ref{fig: unique extremal of F}).
\end{thm}

Next, we focus on the exact Tur\'an number of hypergraphs obtained by blowing up one vertex of $\cF_5$. 

Notice that $\cF_5(f_1;t)$ and $\cF_5(f_2;t)$ are both  subgraphs of $\cF_5^S(t)$, whose Tur\'an number has been determined in \cite{bcl}.
The value of $\cF_5(f_5;t)$ is 
determined in 
Theorem
\ref{thm: F5,f1,t}.
Since $f_3$ and $f_4$ are equivalent in $\cF_5$, $\cF_5(f_3;t)$ is isomorphic to $\cF_5(f_4;t)$.
We determine the exact value of Tur\'an number of $\cF_5(f_3,t)$. Recall that $T_3(n,3)^{t+}$ is obtained by adding any $D(\left\lfloor n/3\right\rfloor,3,t,2)$ or $D(\left\lceil n/3\right\rceil,3,t,2)$ to every part of $T_3(n,3)$. 
We have the following result.
\begin{thm}\label{thm: F5 v3 t}
    When $n$ is sufficiently large, $\ex_3(n,\cF_5(f_3;t))=e(T_3(n,3)^{t+})$.
\end{thm}
Notice that the above theorem shows that the hypergraph $\cF_5(f_3;t)$ also has exponentially many non-isomorphic extremal constructions.

Very recently, Hou et al. \cite{Hou} investigated Tur\'{a}n problems for vertex-disjoint unions of a class of hypergraphs in the degenerate case and obtained near-optimal upper bounds in several cases. In fact, Tur\'{a}n problems for hypergraph forests can all be regarded as Tur\'{a}n problems for disjoint unions of hypergraph trees. In addition, a number of results have already been established on Tur\'{a}n problems for vertex-disjoint unions of expansions. Gu, Li and Shi \cite{GuLS} determined the Tur\'{a}n numbers for the vertex-disjoint union of expansions of cycles in $r$-uniform hypergraphs for $r\geq5$ and sufficiently large $n$.
Gerbner \cite{Ge1} determined the Tur\'{a}n number for the expansion of a graph $F$, where $F$ is the vertex-disjoint union of $k$ components with chromatic number $t+1$, each containing a color-critical edge, and any number of components with chromatic number at most $t$. Here, we consider the vertex-disjoint unions of $\cF_5^S(m)$. 

Let $(t+1)\cdot \cF_5^S(m)$ denote $t+1$ vertex-disjoint copies of $\cF_5^S(m)$.
Here, we give the Tur\'an number of $(t+1) \cdot \cF_5^S(m)$.
\begin{thm}\label{thm: t cF_5^S(m)}
    For every $m\geq 2$, when $n$ is sufficiently large, the following hypergraph $\cG(t,m,n)$ is the extremal hypergraph of $(t+1)\cdot \cF_5^S(m)$ with maximum number of hyperedges.
    
    Fix a set $T$ with $t$ vertices, and a hypergraph $T_3(n-t,3)$ disjoint with $T$.
Let $\cG(t,m,n)$ denote the $3$-uniform $n$-vertex hypergraph obtained by adding all the hyperedges intersecting $T$ to the union of $T$ and $T_3(n-t,3)$. 
\end{thm}
The paper is organized as follows. In Section \ref{sec: notation}, we introduce the notation and the stability lemma of $\cF_5$. In Section \ref{sec: section 3}, we give the proof of Theorem \ref{thm: cF_5^S(m,B)}. We give the proof of Theorem \ref{thm: F5 v8 t} in Section \ref{sec: section 4}, the proof of Theorem \ref{thm: F5 v3 t} in Section \ref{sec: section 5}, and the proof of Theorem \ref{thm: t cF_5^S(m)} in Section \ref{sec: section 6}.



\section{Notation and preliminaries}\label{sec: notation}
In this paper, we use standard terminology and notation (see, e.g., \cite{Bol2}).
Since we focus on $3$-uniform hypergraphs, all hypergraphs in the rest of this paper are $3$-uniform by default.

For a hypergraph $\cH$ (resp. a set $\cE$ of hyperedges), and a set $S\subseteq V(\cH)$ (resp. $S\subseteq \bigcup_{E\in \cE}E$), let $\cH(S)$ (resp. $\cE(S)$) denote the set of hyperedges in $E(\cH)$ (resp. $\cE$) containing $S$, and $d_{\cH}(S)=|\cH(S)|$ (resp. $d_{\cE}(S)=|\cE(S)|$) denote the degree of $S$ in $\cH$ (resp. $\cE$). When $S=\{v\}$, we write $\cH(v)$ (resp. $\cE(v)$) and $d_{\cH}(v)$ (resp. $d_{\cE}(v)$) instead of $\cH(\{v\})$ (resp. $\cE(\{v\})$) and $d_{\cH}(\{v\})$ (resp. $d_{\cE}(\{v\})$). For a set $S\subseteq V(\cH)$, let $\cH[S]$ denote the hypergraph with vertex set $S$ and hyperedges contained in $S$. 

Recall that $L(\cH(v))$ denotes the link graph of $v$ in a $3$-graph $\cH$, which is a graph with vertex set $(\bigcup_{E\in \cH(v)}E)\setminus\{v\}$, and edge set $\{xy: vxy\in E(\cH)\}$.

In this paper, we mainly consider the Tur\'an number of $\cF$, which is a subgraph of a blow-up of $\cF_5$.
For a $3$-uniform hypergraph $\cH$ with a partition $V(\cH)=A\cup B\cup C$ of its vertex set, we call a hyperedge \textbf{good} if it has one vertex in each part.
Let $\cH_{g}$ denote the set of good hyperedges in $\cH$.
For the partition $V(\cH)=A\cup B\cup C$, let $\cK(\cH)$ denote the complete tripartite $3$-uniform hypergraph with parts $A,B,C$.
The hyperedges in $E(\cK)$ but not in $E(\cH)$ are called \textbf{missing} hyperedges and their set is denoted by $\cM(\cH) $. The hyperedges in $E(\cH)$ but not in $E(\cK)$ are called \textbf{extra} hyperedges and their set is denoted by $\cE(\cH) $.
First, let us introduce the stability result of a blow-up of $\cF_5$.
\begin{thm}[\cite{kemu,gowers2001new}]\label{thm: blow-up sta}
    For integers $m,n \geq 1$,
    let $\cF$ be a subgraph of a blow-up of $\cF_5$, and $\cG$ be the maximum $\cF$-free 3-uniform hypergraph on $n$ vertices. 
    For every $\epsilon >0$, there exists $n_0=n_0(\epsilon)$ such that for $n>n_0$, there exists a $3$-partition of $V(\cG)$ with $V(\cG)=A\cup B\cup C$ and $\cK(\cG)=\cK_{A,B,C}$ being the complete tripartite hypergraph, such that the following holds:

   \noindent1. The size of $A$ (resp. $B,C$) is at least $\frac{n}{3}-\epsilon n$ and at most $\frac{n}{3}+\epsilon n$;

   \noindent 2. The size of $\cM $ and $\cE $ defined above satisfy $|\cM(\cG) |\leq  \epsilon n^3$ and $|\cE(\cG) |\leq \epsilon n^3$.
\end{thm}

Let $F, \cG$ and $A,B,C$ be defined as in Theorem \ref{thm: blow-up sta}.
Recall that $\cG_g$ is the set of good hyperedges of $\cG$. Then for $\delta=\epsilon ^{1/3} >0$, let $A'$ be the set of vertices in $A$ contained in a small number of good hyperedges. More precisely,
\begin{equation}\label{eq: def of A'}
    A'=\{v\in A: d_{\cG_g}(v) < |B|\cdot|C|-2\delta n^2\}.
\end{equation}
Since for every $v\in A'$, $d_{\cM }(v)\geq 2\delta n^2$, by Theorem \ref{thm: blow-up sta}, we have
$$\epsilon n^3\ge |\cM |\geq \sum_{v\in A'}d_{\cM }(v)\geq \sum_{v\in A'}2\delta n^2\geq 2\delta n^2|A'|.$$
This implies that $|A'|\leq \frac{1}{2}\delta^2 n$. We can similarly define $B'$ and $C'$, and we have $|B'|\leq \frac{1}{2}\delta^2 n$ and $|C'|\leq \frac{1}{2}\delta^2 n$.
We set $X=A'\cup B'\cup C'$, then
\begin{equation}\label{eq: size of X}
    |X|\leq \frac{3}{2}\delta^2n.
\end{equation}
For a pair of vertices $u,v$ in different parts, we call this pair \textbf{proper} if $d_{\cM }(\{u,v\})\leq \epsilon^{\frac{1}{2}}n$, and \textbf{improper} otherwise.
Note that since $|\cM |\leq \epsilon n^3$, the number of improper pairs is at most $3\epsilon^{\frac{1}{2}} n^2$.

\section{Tur\'an number of \texorpdfstring{$\cF_5^S(m,t,K_{s,s})$}{F\_5^\{S(m,t,K\_\{s,s\})\}} }\label{sec: section 3}
In this section, we prove Theorem \ref{thm: cF_5^S(m,B)}. 
We need the following theorems about complete multipartite graphs and hypergraphs.
\begin{thm}[\cite{kovari1954}]\label{thm: ex of Kss}
For an integer $s\geq 2$, we have
    $\ex(n,K_{s,s})\leq C_1(s)n^{2-\frac{1}{s}}$ for some constant $C_1(s)$.
\end{thm}

\begin{thm}[\cite{geneson2021generalization}]\label{thm: ex3 of Ktss} 
When $s,t\geq 2$,
    $\ex_3(n,K_{t,s,s})\leq C_2(s,t)n^{2-\frac{1}{s^2}}$ for some constant $C_2(s,t)$. When $t=1$ and $s\geq 2$, we have
    $\ex_3(n,K_{1,s,s})\leq C_2'(s)n^{2-\frac{1}{s}}$
    for some constant $C_2'(s)$.
\end{thm}

\begin{thm}[\cite{ZhaoGTAC}]\label{thm: supersat}
    For an integer $s\geq 2$ there exists $C=C(s)>0$, $c=c(s)>0$ and an integer $n_0=n_0(s)$ such that the following holds. For every graph $G$ with $n>n_0$ vertices, if $e(G)\geq C\cdot n^{2-\frac{1}{s}}$, then $G$ contains at least \[c\cdot \left(\frac{e(G)}{\binom{n}{2}}\right)^{s^2}n^{2s}\]
    copies of $K_{s,s}$.
\end{thm}

\begin{thm}\label{thm: ex bip Ktss}
For every $\epsilon>0$ and $s,t\geq 2$, there is a constant $\delta(s,t)$ such that
when $n$ is sufficiently large, and $m\geq \epsilon n$, then
$$\ex^{bip}_3(m,n,K_{t,s,s})\leq \delta(s,t)\epsilon^{-\frac{1}{2}}mn^{2-\frac{1}{s^2}}$$ when $t\geq 2$. And when $t=1$, for every $m$ and sufficiently large $n$, we have
    $$\ex^{bip}_3(m,n,K_{1,s,s})\leq C_3(s) mn^{2-\frac{1}{s}}$$
    for some constant $C_3(s)$.
\end{thm}
\begin{proof}
Suppose that a $3$-uniform hypergraph $\cG$ has a partition $A,B$ with $|A|=m$ and $|B|=n$ and every hyperedge contains one vertex in $A$ and two vertices in $B$. Assume furthermore that there are no $t$ vertices in $A$ such that the intersection of their link graphs in $B$ contains a fixed copy of $K_{s,s}$. Suppose that $\cG$ has the maximum number of hyperedges among all such hypergraphs, and assume that $e(\cG)> \delta(s,t)\epsilon^{-\frac{1}{2}}mn^{2-\frac{1}{s^2}}$, where $\delta(s,t)$ is a constant we will determine later. 

First, we deal with the case where $t\geq 2$. We set $M=\delta(s,t)\cdot \epsilon^{-\frac{1}{2}}$.
We may assume that for every $v\in A$, the number of hyperedges containing $v$ is at least $C\cdot n^{2-\frac{1}{s}}$, where $C$ is described in Theorem \ref{thm: supersat}.
Otherwise, we can delete those hyperedges. In this way, we delete at most $Cmn^{2-\frac{1}{s}}=o(mn^{2-\frac{1}{s^2}})$ hyperedges.  

Then, according to Theorem \ref{thm: supersat}, for every $v\in A$, 
$\cG(v)$ contains at least 
$c(s)\cdot \left(\frac{|\cG(v)|}{{\binom{n}{2}}}\right)^{s^2}n^{2s}$
copies of $K_{s,s}$ where $c(s)$ is described in Theorem \ref{thm: supersat}.

Next, we construct a bipartite graph $H=H(C,D)$, where $C=A$ and $D$ consists of all the $2s$-sets in $B$. A vertex $v\in C$ is adjacent to a $2s$-set $S$ in $D$, if there exists a copy of $K_{1,s,s}$ containing $v$ as the one-vertex part and all the vertices of $S$.

Let $\eta(s)$ denote the number of copies of $K_{s,s}$ in $K_{2s}$. We set $\delta(s,t)=\left(\frac{t\eta(s)}{c(s)}\right)^{\frac{1}{2}}$.
Then we have $M=\delta(s,t)\cdot \epsilon^{-\frac{1}{2}}=\left(\frac{t\eta(s)}{\epsilon c(s)}\right)^{\frac{1}{2}}$,
and the number of edges in $H$ is at least
\begin{align*}
    \sum_{v\in A}c(s)\cdot \left(\frac{|\cG(v)|}{{\binom{n}{2}}}\right)^{s^2}n^{2s}&\geq m\cdot c(s)\cdot \left(\frac{1}{m}\frac{\sum_{v\in A}|\cG(v)|}{{\binom{n}{2}}}\right)^{s^2}n^{2s}\\
    &\geq c(s)m\cdot M^2n^{2s-1}
    \geq \epsilon c(s)M^2\cdot n^{2s}\\
    &>\eta(s)\cdot t\cdot \binom{n}{2s}.
\end{align*}
The first inequality holds because the function $x^{s^2}$ is convex and the second inequality uses that $e(\cG)\geq \delta(s,t)\epsilon^{-\frac{1}{2}}mn^{2-\frac{1}{s^2}}$. The above inequality implies that there is a copy of $K_{s,s}$ inside $B$ that is counted at least $t$ times.
Then there is a copy of $K_{s,s}$ (denoted by $K$), and $t$ vertices $v_1,\dots,v_t\in A$ such that $K$ is a subgraph of $\cG(v_i)$ for $i=1,\dots,t$, a contradiction.

When $t=1$, note that $\cG(v)$ is $K_{s,s}$-free, according to Theorem \ref{thm: ex of Kss}, we have 
$|\cG(v)|\leq C_1(s)\cdot n^{2-\frac{1}{s}}$ for some $C_1(s)$. Then $e(\cG)\leq C_1(s)mn^{2-\frac{1}{s}}$, and we finish the proof by setting $C_3(s)=C_1(s)$.
\end{proof}
With the above results, we have the following inequality.
\begin{lemma}\label{lem: inequality of ex3 bip}
    For every $0<\epsilon<\frac{1}{3}$ and $s,t\geq 2$, there is a constant $\delta(s,t)$ such that
$$\ex^{bip}_3((1/3+\epsilon)n,(1/3+\epsilon)n,K_{t,s,s})-\ex^{bip}_3(n/3 ,n/3,K_{t,s,s})\leq 10\epsilon^{\frac{1}{2}}\delta(s,t)n^{3-\frac{1}{s^2}}.$$
\end{lemma}
\begin{proof}
    Suppose $\cG$ is a $3$-uniform graph with two parts $A,B$ each with size $(1/3+\epsilon)n$, which is the extremal hypergraph of $\ex_3^{bip}((1/3+\epsilon)n,(1/3+\epsilon)n,K_{t,s,s})$.
    Every hyperedge in $\cG$ consists of one vertex in $A$ and two vertices in $B$.

    We may assume $A=\{a_1,\dots,a_{(\frac{1}{3}+\epsilon)n}\}$ and $B=\{b_1,\dots,b_{(\frac{1}{3}+\epsilon)n}\}$, with $d_\cG(a_i)\leq d_{\cG}(a_j)$ and $d_\cG(b_i)\leq d_{\cG}(b_j)$ when $i<j$.
    Suppose $d_\cG(a_{\epsilon n})=x$, and $d_\cG(b_{\epsilon n})=y$.
    Then, we have $e(\cG)\geq \frac{1}{3}xn$, and $e(\cG)\geq \frac{1}{2}y\cdot\frac{1}{3}n$.
    By Theorem \ref{thm: ex bip Ktss}, we have
    $$x\leq \delta(s,t)\epsilon^{-\frac{1}{2}}(1+3\epsilon)(1/3+\epsilon)^{2-\frac{1}{s^2}}n^{2-\frac{1}{s^2}}.$$
    And 
    $$y\leq 6\delta(s,t)\epsilon^{-\frac{1}{2}}(1/3+\epsilon)^{3-\frac{1}{s^2}}n^{2-\frac{1}{s^2}}.$$
    We delete all the vertices in $\{a_1,\dots,a_{\epsilon n}\}\cup \{b_1,\dots,b_{\epsilon n}\}$, the remaining hypergraph (denoted by $\cG'$) is still $K_{t,s,s}$-free,  with each part $\frac{1}{3}n$ vertices. The size of $\cG'$ is bounded by $\ex^{bip}_3(\frac{1}{3}n,\frac{1}{3}n,K_{t,s,s})$.

    Thus, we have
    \begin{align*}
        &\ex^{bip}_3((1/3+\epsilon)n,(1/3+\epsilon)n,K_{t,s,s})-\ex^{bip}_3(n/3,n/3,K_{t,s,s})\\
        &\leq e(\cG)-e(\cG')
        \leq\epsilon n\cdot (x+y)\\
        &\leq \delta(s,t)\epsilon^{\frac{1}{2}}(1+3\epsilon)(1/3+\epsilon)^{2-\frac{1}{s^2}}n^{3-\frac{1}{s^2}}+6\delta(s,t)\epsilon^{\frac{1}{2}}(1/3+\epsilon)^{3-\frac{1}{s^2}}n^{3-\frac{1}{s^2}}\\
        &\leq 10\epsilon^{\frac{1}{2}}\delta(s,t)n^{3-\frac{1}{s^2}}.
    \end{align*}
    The last inequality holds when $\epsilon<\frac{1}{3}$. We complete the proof by setting the constant $\delta(s,t)$ as in Theorem \ref{thm: ex bip Ktss}.
\end{proof}

We can use Theorem \ref{thm: supersat} to obtain the following lemma.
\begin{lemma}\label{lem: find Km,m,m}
    Let $\cH$ be a $3$-uniform hypergraph with $n$ vertices.
    For every $\eta>0$ and integers $m_1,\ell_2\geq 1$, there exist $C=C(\eta,m_1,\ell_2)$ and $n_0=n_0(\eta,m_1,\ell_2)$ such that the following hold for $n>n_0$. Let $A=\{a_1,\dots, a_C\}$ such that for each $i\le C$ we have $|L(\cH(a_i))[V(\cH)-A]|\geq \eta n^2$. Then there exist two disjoint sets $W_1, W_2\subseteq V(\cH)$ of size $m_1$ and a set $A'\subseteq A$ of size $\ell_2$, such that for every $w_1\in W_1$, $w_2\in W_2$ and $a\in A'$, $w_1w_2a\in E(\cH)$.
\end{lemma}

\begin{proof}
    For every $i\le C$, 
    we can find $c'\cdot n^{2m_1}$ copies of $K_{m_1,m_1}$ in $L(\cH(a_i))$ inside $[V(\cH)-A]$ for some constant $c'=c(\eta,m_1)$ by applying Theorem \ref{thm: supersat}. Our goal is to find a copy of $K_{m_1,m_1}$ inside $[V(\cH)-A]$ that belongs to the link graph of at least $\ell_2$ vertices of $A$. Observe that the number of copies of $K_{m_1,m_1}$ is at most $c''n^{2m_1}$ for some constant $c''=c''(m_1)$. Therefore, if $C>\ell_2 c''/c'$, then by the pigeonhole principle, we find the desired copy of $K_{m_1,m_1}$.
\end{proof}
We write $\cF=\cF_5^S(m,t,K_{s,s})$ in the rest of this section. Then we begin the proof of Theorem \ref{thm: cF_5^S(m,B)}. 
We restate Theorem\ref{thm: cF_5^S(m,B)}
here for convenience.
\bigskip

\noindent\textbf{Theorem.}
{\it
When $n$ is large enough, for positive integers $m,t$ and $s$, with $m\geq 2s\geq 4$, 
we have
    $$\ex_3(n,\cF_5^S(m,t,K_{s,s}))\leq e(T_3(n,3))+6\cdot\ex^{bip}_3(n/3,n/3,K_{t,s,s})+3\cdot\ex_3(n/3,K_{t,s,s})+o(n^{3-\frac{1}{s^2}}),$$
    and 
    $$\ex_3(n,\cF_5^S(m,t,K_{s,s}))\geq e(T_3(n,3))+3\cdot\ex_3(\lfloor n/3\rfloor,K_{t,s,s}).$$
    Moreover, when $t=1$,
    $$\ex_3(n,\cF_5^S(m,1,K_{s,s}))\leq e(T_3(n,3))+6\cdot\ex^{bip}_3(n/3,n/3,K_{1,s,s})+3\cdot\ex_3(n/3,K_{1,s,s})+o(n^{3-\frac{1}{s}}).$$
}

For all $t\geq 1$,
the lower bound is given by embedding into each part of $T_3(n,3)$ a $K_{t,s,s}$-free $3$-graph, with the maximum number of hyperedges, the resulting hypergraph is denoted by $\cH$, and the parts of $T_3(n,3)$ will be called the parts of $\cH$. Suppose that $\cH$ contains a copy of $\cF_5^S(m,t,K_{s,s})$. Recall that $m\geq 2s$ and $\cF_5^S(m,t,K_{s,s})$ contains a copy of $T_3(3m,3)$. In this copy of $T_3(3m,3)$, the vertices in $T$ and $K_{s,s}$ cannot be embedded into the same part of $\cH$, because each part is $K_{t,s,s}$-free.
Since there are no hyperedges containing two vertices in a part and one vertex in another, there are two vertices in $\cF_5^S(m,t,K_{s,s})$ in the $K_{s,s}$ part that are embedded into different parts of $\cH$, denote them by $s_1,s_2$.
Then, since $m\ge 2$, there exist two other vertices $u_1,u_2$ such that both $\{s_1,u_1,u_2\}$ and $\{s_2,u_1,u_2\}$ form a hyperedge of $\cH$.
But there are no such two vertices in $\cH$, a contradiction.

Let us turn to the proof of the upper bound.
Let $\cG$ be a maximum $\cF$-free 3-uniform hypergraph on $n$ vertices. 
Let $A,B,C,\delta,\epsilon$ and $\cK=\cK(\cG)$, $\cM=\cM(\cG)$, $\cE=\cE(\cG)$ be defined as in Theorem \ref{thm: blow-up sta}. We set $\delta=\frac{1}{100m}$ for the rest of this section. We define $A',B',C'$ similarly as (\ref{eq: def of A'}) in Section \ref{sec: notation}.
Moreover, we suppose $\cG$ has the maximum number of good hyperedges under the partition $A,B,C$.
We let $A_L'$ be the set of vertices in $A'$ with large degree, which is defined as follows:
$$A_L'=\{v\in A':d_{\cG}(v)\geq |B|\cdot|C|-\delta n^2\}.$$
Note that vertices of $A'$ are in less than $|B|\cdot|C|-2\delta n^2$ good hyperedges, thus vertices of $A_L'$ are in more than $\delta n^2$ hyperedges that are not good.

We can similarly define $B_L'$ and $C_L'$. Let $X_L=A_L'\cup B_L'\cup C_L'$ and recall that $X=A'\cup B'\cup C'$. We will prove the following result.
\begin{lemma}\label{lem: const bound XL}
    There is a constant $N=N(\delta,m,s)$ such that $|X_L|\leq N$.
\end{lemma}
\begin{proof}
    For a vertex $v\in A_L'$, we give the following types of the hyperedges in $\cG(v)$:
    $$\cG_1(v)=\cG_g(v),\, \cG_2(v)=\{E\in \cG(v):|E\cap X|\geq 2\},$$
    $$\cG_3(v)=\{E\in \cG(v):|E\cap (B\setminus B')|=2 \text{~or~} |E\cap (C\setminus C')|=2\text{~or~} |E\cap (A\setminus A')|=2\},$$
    $$\cG_4^B(v)=\{E\in \cG(v):|E\cap (B\setminus B')|=1 \text{~and~} |E\cap (A\setminus A')|=1\},$$
    $$\cG_4^C(v)=\{E\in \cG(v):|E\cap (C\setminus C')|=1 \text{~and~} |E\cap (A\setminus A')|=1\}.$$

    We have $\cG(v)=\cG_1(v)\cup \cG_2(v)\cup \cG_3(v)\cup \cG_4^B(v)\cup \cG_4^C(v)$. By the definition of $A'$ and $A_L'$, we have
    \begin{equation}\label{eq: sum of g2g3g4}
        |\cG_2(v)|+|\cG_3(v)|+|\cG_4^B(v)|+|\cG_4^C(v)|\geq \delta n^2.
    \end{equation}
    We can easily bound the size of $\cG_2(v)$, with $|\cG_2(v)|\leq |X|n\leq \frac{3}{2}\delta^2 n^2$ by (\ref{eq: size of X}). 
    
\begin{claim}\label{claim: bound g3}
    There is a constant $N_1=N_1(m,s)$ such that the number of vertices $v\in A_L'$ satisfying
    $$|\cG_3(v)|\geq \frac{1}{10}\delta n^2$$
    is at most $N_1$.
\end{claim}
\begin{proof}
We define $\cG_3^B(v)=\{E\in \cG_3(v):|E\cap (B\setminus B')|=2\}$, and we define $\cG_3^C(v)$ and $\cG_3^A(v)$ similarly. We have $|\cG_3(v)|=|\cG_3^B(v)|+|\cG_3^C(v)|+|\cG_3^A(v)|$. Thus, it is enough to prove that there is a constant $N_2=N_2(m,s)$ such that the set
$$D_1^B=\{v\in A_L':|\cG_3^B(v)|\geq \frac{1}{30}\delta n^2\}$$ has size at most $N_2$.

The size of $B'$ is at most $\frac{1}{2}\delta^2n$, thus, for every $v\in D_1^B$, the number of hyperedges in $\cG_3^B(v)$ that avoid $B'$ is at least $\frac{1}{30}\delta n^2-\frac{1}{2}\delta^2n^2\geq \frac{1}{60}\delta n^2$ when $\delta=\frac{1}{100m}$. 

When $n$ is large enough and $|D_1^B|\geq C(\frac{1}{60}\delta,s,t)$, then by Lemma \ref{lem: find Km,m,m} there exists a set of $t$ vertices (denoted by $D_1^{B,1}$) in $D_1^B$ such that for all the vertices $v\in D_1^{B,1}$, the link graph $L(\cG_3(v))$ is a fixed copy of $K_{s,s}$ in $B\setminus B'$. Denote the vertices of the $K_{s,s}$ as $\{b_1,\dots,b_{2s}\}$. We choose $m-2s$ vertices from $B\setminus B'$ distinct from $\{b_1,\dots,b_{2s}\}$ and denote them as $\{b_{2s+1},\dots,b_{m}\}$.

According to the definition of $B'$, for each $b_i$ $(1\leq i\leq m)$, there are at most $2\delta n^2$ pairs $(a,c)\in A\times C$ such that $a b_i c\notin E(\cG)$. It implies that at least $|A||C|-3m\delta n^2$ pairs $(a,c)\in (A\setminus(D_1^{B,1}))\times C$ satisfy that $a b_i c\in E(\cG)$ for all $1\leq i\leq m$.
We construct an auxiliary bipartite graph $H$ with parts $A$ and $C$, where $a\in (A\setminus D_1^{B,1})$ and $c\in C$ are connected if and only if $a b_i c\in E(\cG)$ for all $1\leq i\leq m$.
Then, we have $e(H)\geq |A||C|-3m\delta n^2\geq (\frac{n}{3})^2-\frac{1}{20}n^2$ when $\delta=\frac{1}{100m}.$
By Theorem \ref{thm: ex of Kss}, there exists a $K_{m,m}$ in $H$, which implies that there are $m$ vertices $a_1,a_2,\dots,a_m$ in $A\setminus D_1^{B,1}$ and $m$ vertices $c_1,c_2,\dots,c_m$ in $C$ such that $a_i b_j c_k\in E(\cG)$ for all $1\leq i,j,k\leq m$. Together with $D_1^{B,1}$, this forms a copy of $\cF$ in $\cG$, a contradiction (see Figure \ref{fig:D1B1 forms FS}). 
\begin{figure}
    \centering
\includegraphics[width=0.6\linewidth]{ 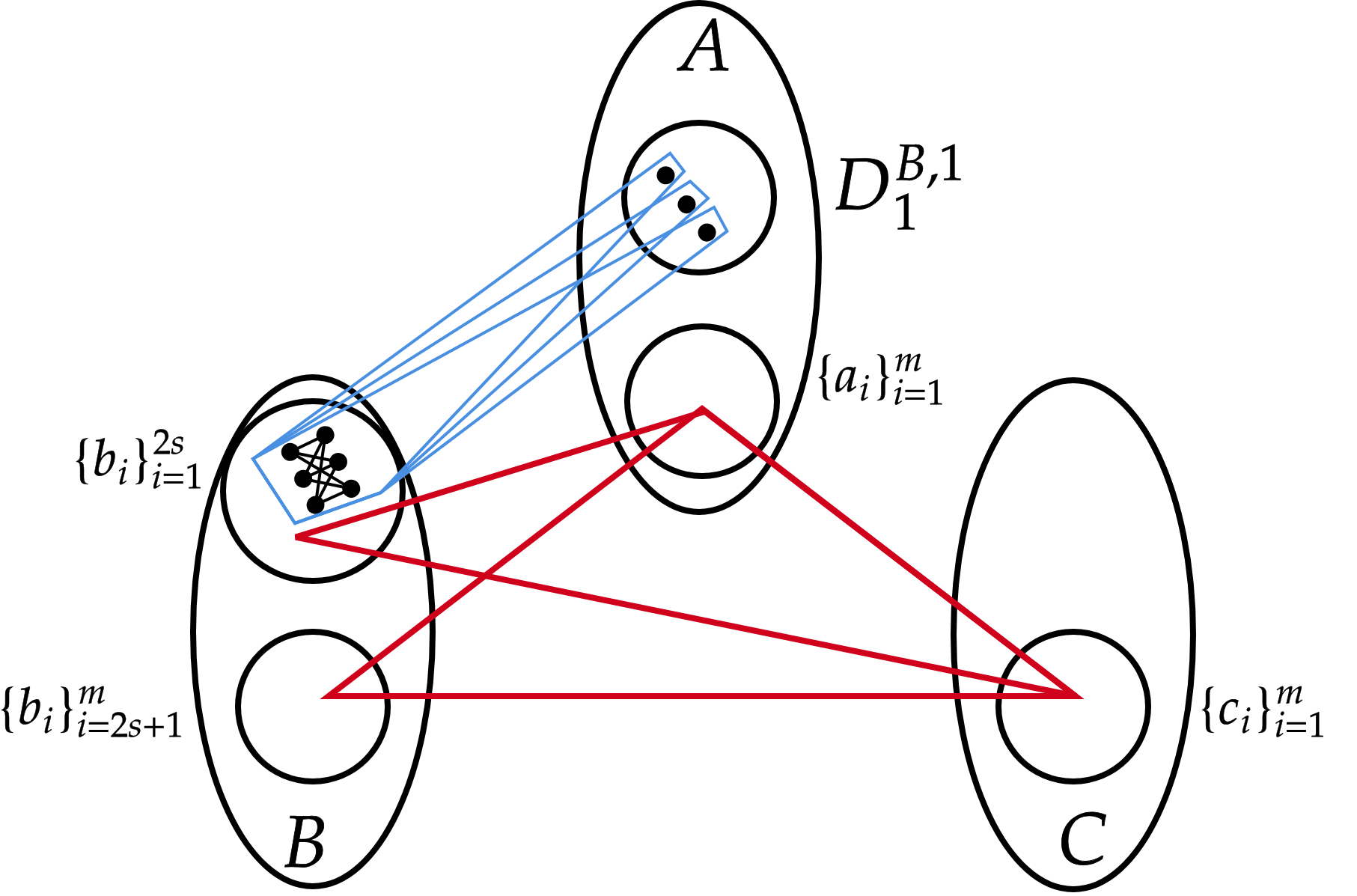}
    \caption{The vertices $\{a_i\}_{i=1}^{m}\cup \{b_i\}_{i=1}^{m}\cup\{c_i\}_{i=1}^{m}$ and $D_{1}^{B,1}$ form a copy of $\cF$}
    \label{fig:D1B1 forms FS}
\end{figure}

Then, by choosing $N_2=C\left(\frac{1}{60}\delta,s,t\right)$, the claim holds.
\end{proof}
Let us delete the at most $N_1$ vertices described in Claim \ref{claim: bound g3} from $A_L'$ and let $R_1$ denote the resulting set, i.e. the set of vertices $v\in A_L'$ with $|\cG_3(v)|< \frac{1}{10}\delta n^2$.
According to (\ref{eq: sum of g2g3g4}) and the bounds on $|\cG_2(v)|$, for every $v\in R_1$, we have
$$|\cG_4^B(v)|+|\cG_4^C(v)|\geq\delta n^2 -|\cG_2(v)|-|\cG_3(v)|\geq  \delta n^2 - \frac{3}{2}\delta^2 n^2 - \frac{1}{10}\delta n^2\geq \frac{\delta}{2} n^2,$$
when $n$ is sufficiently large and $\delta=\frac{1}{100m}$. Without loss of generality, it is sufficient to prove that there exists a constant $N_3=N_3(\delta,m)$ such that there are at most $N_3$ vertices $v\in R_1$ satisfying that
$$e\left(L\left(\cG_4^B(v)\right)\right)=|\cG_4^B(v)|\geq \frac{\delta}{4} n^2.$$
Let us denote the set of such vertices (the vertices $v\in R_1$ with $|\cG_4^B(v)|\geq \frac{\delta}{4}n^2$) in $R_1$ as $D_2^B$.
\begin{claim}
    There is a constant $N_3=N_3(\delta,m)$ such that $|D_2^B|\leq N_3$.
\end{claim}
\begin{proof}
According to Lemma \ref{lem: find Km,m,m}, when $|D_2^B|\geq N_3'(\delta,m+t,\ell_2)$, where $N_3'(\delta,m+t,\ell_2)$ is the constant defined in Lemma \ref{lem: find Km,m,m} and $\ell_2=\ell_2(m)$ is a constant we will determine later, there exist two disjoint sets $W_1\subseteq A\setminus A'$, $W_2\subseteq B\setminus B'$ both of size $m+t$, and a set $D_2^{B,1}\subseteq D_2^B$, of size $\ell_2$, such that for every $w_1\in W_1$, $w_2\in W_2$, $w_1 w_2 a\in E(\cG)$ for every $a\in D_2^{B,1}$.
Since $m\geq 2s$,
we choose a set $T\subseteq W_2$ with $t$ vertices such that the link graph of every $w_2\in T$, $L\left(\cG[\{w_2\}\cup W_1\cup D_2^{B,1}](w_2)\right)$, contains a copy of $K_{m,\ell_2}$, denoted by $K'$.

Note that we assume $\cG$ has the maximum number of good hyperedges  among all the maximum $\cF$-free 3-uniform hypergraphs on $n$ vertices such that each hyperedge contains one vertex from each of $A, B$ and $C$. 
 Then for every $v\in D_2^{B,1}$, $d_{\cG_g}(v)=|\cG_1(v)|\geq |\cG_4^B({v})|$ and $d_{\cG_g}(v)\geq |\cG_4^C(v)|$. 

Note that $$|\cG_1(v)|+|\cG_4^B({v})|+|\cG_4^C(v)|\geq |B|\cdot |C|-2\delta n^2-|\cG_3(v)|-|\cG_2(v)|\geq |B|\cdot |C|-3\delta n^2,$$
which implies that $|\cG_1(v)|\geq \frac{1}{3}(|B|\cdot |C|-3\delta n^2)$.

According to the definition of $A'$, for each $u\in  W_1$, there are at most $2\delta n^2$ pairs $(b,c)\in B\times C$ such that $u b c\not\in E(\cG)$. It implies that at least $|B||C|-2m\delta n^2$ pairs $(b,c)\in B\times C$ satisfy that $u b c\in E(\cG)$ for all $u\in  W_1$. 
We construct an auxiliary bipartite graph $H''$ with parts $B$ and $C$, where $b\in B$ and $c\in C$ are connected if and only if $u b c\in E(\cG)$ for all $u\in  W_1$.
Then we construct an auxiliary $3$-uniform hypergraph $\cH''(D_2^{B,1})$ with vertex set $D_2^{B,1}\cup B\cup C$, where $a\in D_2^{B,1}$, $b\in B$ and $c\in C$ form a hyperedge if and only if $b c\in E(H'')$.

Notice that for every $a\in D_2^{B,1}$, $$d_{\cH''(D_2^{B,1})}(a)\geq \cG_1(a)-2m\delta n^2\geq \frac{1}{3}(|B|\cdot |C|-3\delta n^2)-2m\delta n^2\geq \frac{1}{4}(\frac{n}{3})^2.$$
Again, there is a constant $N_3''=N_3''(1/36,m,m)$ such that when $\ell_2\geq N_3''$, by Lemma \ref{lem: find Km,m,m}, there exist two disjoint sets $W_B\subseteq B$, $W_C\subseteq C$ both of size $m$, and a set $D_2^{B,2}\subseteq D_2^{B,1}$ of size $m$, such that for every $w_b\in W_B$, $w_c\in W_C$, $a w_b w_c\in E(\cH''(D_2^{B,1}))$ for every $a\in D_2^{B,2}$, and $w_b w_c\in E(H'')$.
Here, $w_b w_c\in E(H'')$ implies that $u w_b w_c\in E(\cG)$ for every $u\in W_1$.

Now, for the fixed set $T$, 
there is a copy of $K_{s,s}$ contained in $A\cap L(\cG[\{w_2\}\cup W_1\cup D_2^{B,2}](w_2))$ for every $w_2\in T$. Then there exists a copy of $\cF$ in $\cG$ contained in $T\cup W_1\cup D_2^{B,2}\cup W_B\cup W_C$, a contradiction (see Figure \ref{fig:D1B2 contains FS}).

\begin{figure}
    \centering
\includegraphics[width=0.6\linewidth]{ 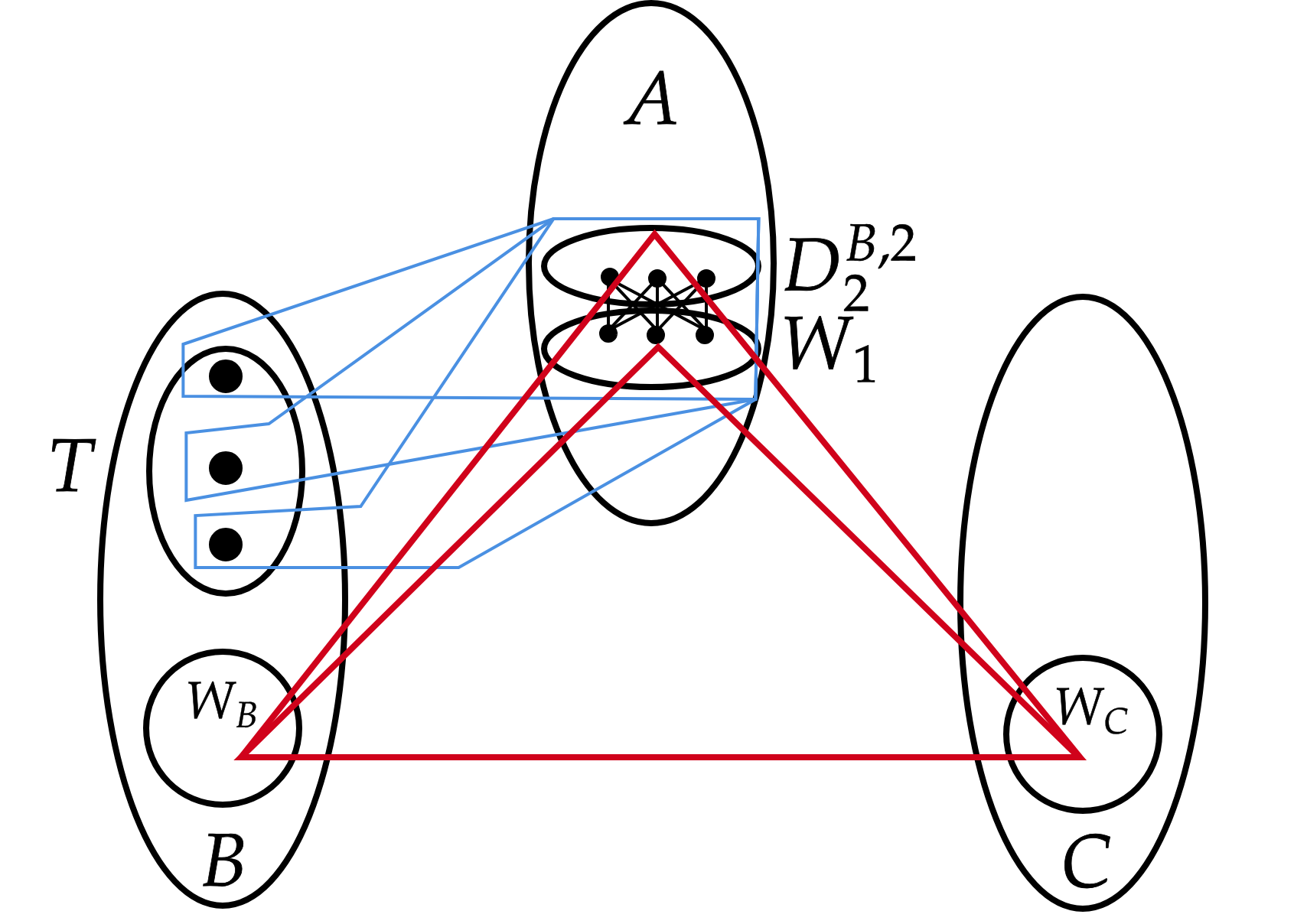}
    \caption{There exist $s$ vertices in $D_{2}^{B,2}$ and $s$ vertices in $W_1$, together with $T\cup W_B\cup W_C$, they form a copy of $\cF$.}
    \label{fig:D1B2 contains FS}
\end{figure}
The claim holds by choosing $N_3=N_3'(\delta, m+t, l_2)$. 
\end{proof}
By symmetry, we also have the same bound for the number of vertices in $R_1$ with $|\cG_4^C(v)|\ge \frac{\delta}{4}n^2$, thus there exists a constant $N_4=N_4(\delta,m)$ such that $|A'_L| \leq N_4$ and similarly $|B'_L| \leq N_4$, $|C'_L| \leq N_4$.
Thus, we can choose $N=3N_4$ such that the lemma holds.
\end{proof}

In the following, we set 
$$h(n)= e(T_3(n,3))+6\cdot\ex^{bip}_3(n/3,n/3,K_{t,s,s})+3\ex_3(n/3,K_{t,s,s})+o(n^{3-\frac{1}{s^2}}).$$
For every $v\in  A'\setminus A_L'$, 
we have $d_{\cG}(v)<|B|\cdot|C|-\delta n^2< \left(\frac{1}{9}+\epsilon-\delta\right)n^2.$ 
By Lemma \ref{lem: const bound XL}, the number of hyperedges in $\cG$ containing at least one vertex in $X$ is at most
$$|X\setminus X_L|\cdot \left(\left(\frac{1}{9}+\epsilon-\delta\right)n^2\right)+|X_L|n^2=|X|\cdot\left(\left(\frac{1}{9}+\epsilon-\delta\right)n^2\right)+O(n^2).$$
Then we focus on the hyperedges in $\cG$ not containing any vertex in $X$. Let $\cG'=\cG[V(\cG)\setminus X]$, and $|V(\cG')|=|V(\cG)|-|X|$.
\begin{claim}
   The hypergraph $\cG'[A\setminus A']$ (resp. $\cG'[B\setminus B']$, $\cG'[C\setminus C']$) is $K_{t,s,s}$-free.
\end{claim}
\begin{proof}
    We only prove the statement for $\cG'[A\setminus A']$, the other two cases can be proved similarly.
    Suppose there is a $K_{t,s,s}$ in $\cG'[A\setminus A']$. That is, there exist $t$ vertices $u_1,\dots,u_t\in A\setminus A_L'$ and $2s$ vertices $a_1,a_2,\dots,a_{2s}$ in $A\setminus A'$ which forms a $K_{t,s,s}$. 
    Then, analogous to the proof of Claim \ref{claim: bound g3}, there exists a copy of $\cF$ in $\cG$, a contradiction.
\end{proof}
Thus, the number of hyperedges in $\cG'$ that are contained in one of the three parts is at most
$$f_1(A,B,C,X):=\ex_3(|A\setminus A'|,K_{t,s,s})+\ex_3(|B\setminus B'|,K_{t,s,s})+\ex_3(|C\setminus C'|,K_{t,s,s}).$$
\begin{claim}
    For every $t$-set $T\subseteq A\setminus A'$, the link graph $\bigcap_{v\in T}L(\cG'(v))[B\setminus B']$ is $K_{s,s}$-free.
\end{claim}
\begin{proof}
    The proof is similar as in the proof of Claim \ref{claim: bound g3}, we omit it here.
\end{proof}
By symmetric arguments, the number of hyperedges in $\cG'$ containing exactly two vertices in one of the three parts is at most
\begin{align*}  
    f_2(A,B,C,X):=& \ex^{bip}_3(|A\setminus A'|,|B\setminus B'|,K_{t,s,s})+\ex^{bip}_3(|A\setminus A'|,|C\setminus C'|,K_{t,s,s})\\
    + &\ex^{bip}_3(|B\setminus B'|,|A\setminus A'|,K_{t,s,s}) +\ex^{bip}_3(|B\setminus B'|,|C\setminus C'|,K_{t,s,s})\\
    +&\ex^{bip}_3(|C\setminus C'|,|A\setminus A'|,K_{t,s,s})+\ex^{bip}_3(|C\setminus C'|,|B\setminus B'|,K_{t,s,s}).
\end{align*}
The number of good hyperedges in $\cG'$ is at most
$$f_3(A,B,C,X):=|A\setminus A'|\cdot |B\setminus B'|\cdot |C\setminus C'|\leq e(T_3(n-|X|,3)).$$
As a result, the number of hyperedges in $\cG'$ is at most $$f_1(A,B,C,X)+f_2(A,B,C,X)+f_3(A,B,C,X).$$
The number of hyperedges in $\cG$ is at most
$$f_1(A,B,C,X)+f_2(A,B,C,X)+f_3(A,B,C,X)+|X|(\frac{1}{9}+\epsilon-\delta)n^2+O(n^2).$$
We set $$f(n,\epsilon)=e(T_3(n,3))+6\cdot\ex^{bip}_3((1/3+\epsilon)n,(1/3+\epsilon)n,K_{t,s,s})+3\ex_3((1/3+\epsilon)n,K_{t,s,s}).$$\label{eq: f(n,epsilon)}
\begin{claim}
    The size of $\cG$ is at most $f(n,\epsilon)+O(n^2)$.
\end{claim}
\begin{proof}
   First, we note that
   $$f_3(A,B,C,X)+|X|\left(\frac{1}{9}+\epsilon-\delta\right)n^2\leq e(T_3(n-|X|,3))+|X|\left(\frac{1}{9}+\epsilon-\delta\right)n^2\leq e(T_3(n,3)).$$
   By the monotonically increasing property of $\ex^{bip}_3(n,m,K_{s,s})$ and $\ex_3(n,K_{t,s,s})$, we have
    $$f_1(A,B,C,X)\leq 3\ex_3\left(\left(1/3+\epsilon\right)n,K_{t,s,s}\right),$$
    and
   $$f_2(A,B,C,X)\leq 6\cdot\ex^{bip}_3((1/3+\epsilon)n,(1/3+\epsilon)n,K_{t,s,s}).$$
   Thus, the claim holds.
\end{proof}

Then, it is sufficient to prove that
$$f(n,\epsilon)-h(n)=o(n^{3-\frac{1}{s^2}}),$$
which is equivalent to proving that there is a function $g(\epsilon)$ of $\epsilon$, such that for every $\epsilon>0$, for sufficiently large $n$,
\begin{equation}\label{eq: goal to prove}
    \frac{f(n,\epsilon)-h(n)}{n^{3-\frac{1}{s^2}}}\leq g(\epsilon),
\end{equation}
and $\lim_{\epsilon\to 0}g(\epsilon)=0$.
Notice that
\begin{equation}\label{eq: bip eq1}
\begin{aligned}
    &\ex^{bip}_3((1/3+\epsilon)n,(1/3+\epsilon)n,K_{t,s,s})-\ex^{bip}_3(n/3,n/3,K_{t,s,s})\leq 10\epsilon^{\frac{1}{2}}\delta(s,t)n^{3-\frac{1}{s^2}}
\end{aligned}
\end{equation}
by Lemma \ref{lem: inequality of ex3 bip}.
We have
\begin{equation}\label{eq: bip eq2}
\begin{aligned}
   & \ex_3((1/3+\epsilon)n,K_{t,s,s})-\ex_3(n/3,K_{t,s,s})\\
    \leq& \ex_3(\epsilon n,K_{t,s,s})+\ex^{bip}_3(\epsilon n,n/3,K_{t,s,s})+\ex^{bip}_3(n/3,\epsilon n,K_{t,s,s})\\
    \leq  &C_2(s,t)(\epsilon n)^{3-\frac{1}{s^2}}+2\delta(s,t)\epsilon^{\frac{1}{2}}n^{3-\frac{1}{s^2}},
\end{aligned}
\end{equation}
where $C_2(s,t)$ is from Theorem \ref{thm: ex3 of Ktss}, and $\delta(s,t)$ is from Theorem \ref{thm: ex bip Ktss}.
Combining (\ref{eq: bip eq1}) and (\ref{eq: bip eq2}), and setting $g(\epsilon)=12\delta(s,t)\epsilon^{\frac{1}{2}}+C_2(s,t)\epsilon^{3-\frac{1}{s^2}}$,
it satisfies (\ref{eq: goal to prove}) and $\lim_{\epsilon\to 0}g(\epsilon)=0$.

When $t=1$, we have
$$\ex^{bip}_3(\epsilon n,n,K_{1,s,s})\leq C_3(s)n^{3-\frac{1}{s}},$$
$$\ex^{bip}_3(n,\epsilon n,K_{1,s,s})\leq C_3(s)n^{3-\frac{1}{s}},$$
and 
$$\ex_3(\epsilon n,K_{1,s,s})\leq C_2'(s)(\epsilon n)^{3-\frac{1}{s}},$$
where $C_3(s)$ is from Theorem \ref{thm: ex bip Ktss}, and $C_2'(s)$ is from Theorem \ref{thm: ex3 of Ktss}.
Thus, with a similar proof, we are done.

\section{Tur\'an number of \texorpdfstring{$\cF_{sim}(t)$}{F\_\{sim\}(t)}}\label{sec: section 4}

In this section, we study the Tur\'an number of the hypergraph $\cF_{sim}(t)$ (see Figure \ref{fig: F5 sim}).
Note that $\cF_{sim}(t)$ is contained in a blow-up of $\cF_5$.
For convenience, we write $\cF = \cF_{sim}(t)$ in the rest of this section.
We will prove Theorem \ref{thm: F5 v8 t}, which we restate here for convenience.

\medskip

\noindent\textbf{Theorem.}
    {\it When $n$ is sufficiently large, we have
    \[
        \ex_3(n, \cF) = \max_{a+b+c =n-t} \left\{ abc + t \left( \binom{a}{2} + \binom{b}{2} + \binom{c}{2} \right) \right\} + \binom{t}{3},
    \]
    where $a,b,c$ are positive integers.
    The maximum is achieved when $a,b,c$ differ by at most $1$.
    Equality holds if and only if the extremal hypergraph is obtained from the complete tripartite hypergraph on $n-t$ vertices $T_3(n-t,3)$ by adding $t$ vertices forming a clique, and each vertex of this clique with each pair that lies in one part of the $T_3(n-t,3)$ forms a hyperedge.
}

\smallskip

\begin{proof}
Let $\cG$ be the maximum $\cF$-free 3-uniform hypergraph on $n$ vertices. 
Let $A,B,C,\delta,\epsilon$ and $\cK=\cK(\cG)$, $\cM=\cM(\cG)$, $\cE=\cE(\cG)$ be defined as in Theorem \ref{thm: blow-up sta}. In this section, we set $\delta=\frac{1}{10t}$.
Let $A',B',C'$ be defined as in Section \ref{sec: notation}. We suppose that $\cG$ has the maximum number of good hyperedges among the partition $A,B,C$.
Let $X=A'\cup B'\cup C'$.
We set $$\cE_1=\{E\in \cE:\text{$E\subseteq (A\cup B\cup C)\setminus X$}\}.$$

\begin{claim}\label{claim: empty V_i'}
    We have $\cE_1=\emptyset$.
\end{claim}
\begin{proof}
First, suppose that there is a hyperedge $E=a_1a_2a_3$ with $a_i \in A \setminus A'$ for $i=1,2,3$.
Since $a_1,a_2 \in A\setminus A'$, by the definition of $A'$, 
$d_\cM(a_i)\leq 2\delta n^2$.
Then, there exists a pair $b\in B$ and $c\in C$ such that $a_1bc$, $a_2bc\in E(\cG)$.

Similarly, since $a_3 \in A\setminus A'$, and the number of improper pairs is at most $3\epsilon^{\frac{1}{2}} n^2$,
there exists a proper pair $b_1\in B,c_1\in C$ disjoint from $b,c$ such that $a_3b_1c_1\in E(\cG)$.

Since $b_1c_1$ is a proper pair, there exists $t$ vertices $a_4,\ldots,a_{t+3}$ in $A\setminus A'$ disjoint from $a_1,a_2,a_3$ such that $a_ib_1c_1\in E(\cG)$ for each $i=4,\ldots,t+3$.
Then the vertices $\{a_1,a_2,a_3,b,c,b_1,\\c_1,a_4,\dots,a_{t+4}\}$ span a copy of $\cF$~(see Figure~\ref{fig: the case when}), a contradiction.
 \begin{figure}[h]
        \centering
\includegraphics[width=0.6\textwidth]{ 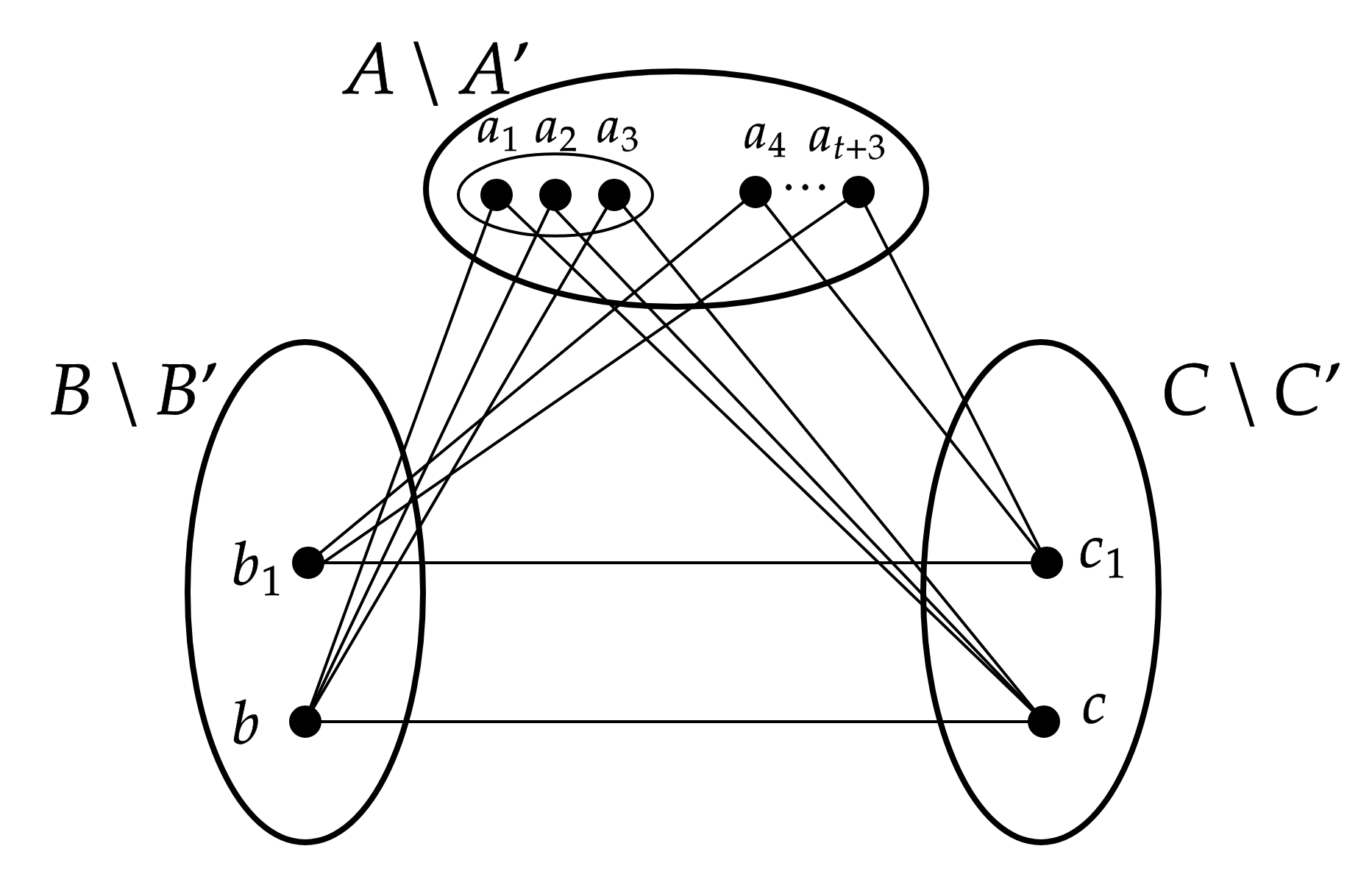}
        \caption{The case when there exists a hyperedge containing three vertices from $A\setminus A'$.}
        \label{fig: the case when}
    \end{figure}
When there exists a hyperedge $E\subseteq B\setminus B'$ or $E\subseteq C\setminus C'$, we can similarly find a copy of $\cF$.
When there is a hyperedge $a_1a_2b$ with $a_1a_2\in A \setminus A'$ and $b \in B\setminus B'$, we can use a similar argument to find a copy of $\cF$ in $\mathcal{H}$.
The details are omitted here. By symmetric arguments, the claim holds.
\end{proof}

Let us consider the hyperedges that contain vertices in $X$.
For a pair of vertices $a,b$, if $a,b$ are from different parts, then we call it a \textbf{good pair}, and we call it a \textbf{bad pair} if they are from the same part.
\begin{claim}\label{claim: X property}
    Let $x$ be a vertex in $X$.
    Suppose that there exists a hyperedge $xyz$ such that $\{y,z\} \subseteq A\setminus A'$ or $\{y,z\} \subseteq B\setminus B'$ or $\{y,z\} \subseteq C\setminus C'$.
    Then for any good pair with $u, w \notin \{y,z\}$, if $uw$ is proper, then the hyperedge $xuw\not\in E(\cG)$. Furthermore, the number of hyperedges containing $x$ and a good pair is at most $4\epsilon^{\frac{1}{2}}n^2$.
\end{claim}

\noindent
\begin{proof}
For the first statement, suppose otherwise, we may assume $y,z \in A\setminus A'$, then there exists a good pair $b,c$ with $b\in B,c\in C$ and $b,c \notin \{x,y,z,u,w\}$, such that $ybc,zbc\in E(\cG)$.
Furthermore, since $uw$ is a proper pair, there exists $t$ vertices $a_1,\ldots,a_t \notin \{x,y,z,u,w,b,c\}$  such that $a_i uw\in E(\cG)$ for each $i=1,2,\ldots,t$.
Then the vertices $\{x,y,z,u,w,b,c\}$ and $\{a_1,a_2,\ldots,a_t\}$ span a copy of $\cF$, a contradiction.

The hyperedges containing $x$ and a good pair $u_1,u_2$ satisfy the following condition:
Either the good pair $u_1,u_2$ is improper,
or $\{u_1,u_2\}\cap \{b,c\}\neq \emptyset$.
Since the number of improper pairs is at most $3\epsilon^{\frac{1}{2}}n^2$,
and the number of good pairs intersecting with $\{b,c\}$ is at most $2n$,
the number of hyperedges containing $x$ and a good pair is at most $3\epsilon^{\frac{1}{2}}n^2+2n<4\epsilon^{\frac{1}{2}} n^2$.
    
\end{proof}

Let us partition the vertices in $X$ into three types.
Let $x$ be a vertex in $X$,
\begin{enumerate}
    \item if there exists a pair $y,z$ from $A\setminus A'$ or $B\setminus B'$ or $C\setminus C'$, and $t$ other vertices $x_1,\ldots, x_t$ from $X$, such that $xyz\in E(\cG)$ and $x_iyz\in E(\cG)$ for each $i=1,\ldots,t$, then we say that $x$ is of \textbf{Type 1};
    \item if there is no hyperedge containing $x$ and two vertices $u_1,u_2$ such that $\{u_1,u_2\}\subseteq A\setminus A'$ or $\{u_1,u_2\}\subseteq B\setminus B'$ or $\{u_1,u_2\}\subseteq C\setminus C'$, then we say $x$ is of \textbf{Type 2};
    \item otherwise, we say $x$ is of \textbf{Type 3}.
\end{enumerate}
Let $X_1, X_2, X_3$ be the set of vertices of Type 1, Type 2 and Type 3, respectively.
For a vertex $x$, recall that $d_\cG(x)$ is the number of hyperedges in $E(\cG)$ containing $x$.
\begin{claim}\label{claim: X1}
    For every $x \in X_1$, we have $d_{\cG}(x)\le 5\delta^{\frac{3}{2}}n^2$.
\end{claim}
\begin{proof}
Since $x \in X_1$, there exists a pair $y,z$ with $\{y,z\}\subseteq A\setminus A'$ or $\{y,z\}\subseteq B\setminus B'$ or $\{y,z\}\subseteq C\setminus C'$, and $t$ other vertices $x_1,\ldots, x_t$ from $X$, $xyz,x_iyz\in E(\cG)$ for each $i=1,\ldots,t$.
The hyperedges containing $x$ can be divided into three types: 
\begin{enumerate}
    \item those containing another vertex from $X$, with at most $\frac{3}{2}\delta^2 n^2$ such hyperedges by (\ref{eq: size of X});
    \item those containing $x$ and a pair of vertices in different parts outside $X$, with at most $4\epsilon^{\frac{1}{2}} n^2=4\delta^{\frac{3}{2}}n^2$ such hyperedges by Claim~\ref{claim: X property};
    \item those containing $x$ and a bad pair outside $X$, that is, two vertices from $A\setminus A'$ or $B\setminus B'$ or $C\setminus C'$.
\end{enumerate}
If there is a hyperedge containing $x$ and a bad pair $y',z'$, we may assume $y',z' \in A\setminus A'$.
If $y',z' \notin \{y,z\}$, then by the definition, there exists a pair $b\in B,c\in C$ with $b,c\notin \{x,y,z,y',z',x_1,\ldots,x_t\}$, such that $y'bc,z'bc\in E(\cG)$.
Then the vertices $\{x,y,z,y',z',b,c,x_1,\\ x_2,\ldots,x_t\}$ span a copy of $\cF$, a contradiction.

That is, the bad pair $y',z'$ must contain at least one vertex from $\{y,z\}$.
There are at most $2n$ such bad pairs, and then the claim follows.

\end{proof}  
\begin{claim}\label{claim: X2}
    If $x \in X_2$, then $d_{\cG}(x) \le \frac{1}{9}n^2 + 3\epsilon n^2 -2 \delta n^2$.
\end{claim}
\begin{proof}
First, the number of hyperedges containing $x$ and another vertex from $X$ is at most $\frac{3}{2}\delta^2 n^2$.

Then we consider the hyperedges containing $x$ and two vertices from $V(\mathcal{\cG}) \setminus X$. 
Let $L^\star_x$ be the link graph of $x$ in $V(\cG) \setminus X$.
Since $x \in X_2$, by the definition of $X_2$, $L_x^\star$ is a tripartite graph with parts $A\setminus A',B\setminus B', C\setminus C'$.
Without loss of generality, assume $x \in A$.
Let $N_{AB}$ (resp. $N_{BC},N_{AC}$) be the number of edges in $L^\star_x$ between $A$ and $B$ (resp. $B$ and $C$, $A$ and $C$).
Then we have $N_{BC} \ge N_{AC},N_{BC}\geq N_{AB}$.
Otherwise, if $N_{AC} > N_{BC}$, we can move $x$ to $B$, which increases the number of good hyperedges, contradicting the choice of the partition.
Moreover, we have $N_{BC} \le |B|\cdot|C|-2\delta n^2$, by the definition of $A'$.
Then we consider two cases.

\textbf{Case 1.}
There is an edge in $L^\star_x$ between $A\setminus A'$ and $B\setminus B'$ or $A\setminus A'$ and $C \setminus C'$.

Without loss of generality, suppose there is an edge $v_av_b\in L^\star_x$ with $v_a\in A\setminus A'$ and $v_b\in B\setminus B'$.

Then assume that there is an edge $u_b u_c$ in $L^\star_x$ disjoint with $v_a,v_b$ where $u_b \in B\setminus B'$ and $u_c \in C\setminus C'$.
We claim that $v_au_bu_c$ is not a hyperedge in $E(\cG)$.
Otherwise, there exists a copy of $\cF$ in $\mathcal{H}$
where $v_b$ plays the role of $f_5$ in $\cF$.
Actually, since $v_b \in B\setminus B'$, by the definition, there exists a good proper pair $w_a w_c$ with $w_a \in A\setminus A', w_c \in C\setminus C'$, disjoint from $v_a,v_b,u_b,u_c$, such that $w_av_b w_c\in E(\cG)$.
Then, there exists $t$ vertices $b_1,\ldots,b_t$ in $B\setminus B'$ such that $b_i w_a w_c\in E(\cG)$ for each $i=1,2,\ldots,t$.
Then the vertices $\{x,v_a,v_b,u_b,u_c,w_a,w_c,b_1,b_2,\ldots,b_t\}$ span a copy of $\cF$, a contradiction.

Since $v_a\in A\setminus A'$, the number of good pairs $u_b,u_c$ between $B\setminus B'$ and $C\setminus C'$ such that $v_au_bu_c$ is not a hyperedge is at most $2\delta n^2$. And the number of pairs $u_b,u_c$ intersecting with $v_a,v_b$ is at most $n$.
Then the number of edges in $L^\star_x$ between $B\setminus B'$ and $C\setminus C'$ is at most $n + 2\delta n^2$, in this case. Then $$d_\cG(x)\leq N_{AC}+N_{BC}+N_{AB}+\frac{3}{2}\delta^2n^2\leq 3N_{BC}+\frac{3}{2}\delta^2n^2\leq 3n+6\delta n^2+\frac{3}{2}\delta^2n^2<\frac{1}{9}n^2+3\epsilon n^2-2\delta n^2$$
when $\delta=\frac{1}{10t}$ and $n$ is sufficiently large.

\textbf{Case 2.}
There is no edge in $L^\star_x$ between $A\setminus A'$ and $B\setminus B'$ or $A\setminus A'$ and $C\setminus C'$.

Then the number of edges in $L^\star_x$ in $V(\mathcal{H}) \setminus X$ is exactly $N_{BC} \le |B|\cdot|C| -2\delta n^2$.

As a conclusion, the number of hyperedges containing $x$ is at most $\epsilon n^2 + |B||C| - \delta n^2 \le \frac{1}{9}n^2 + 3\epsilon n^2 - 2\delta n^2$.
\end{proof}  

Now we consider the vertices in $X_3$.

\begin{claim}\label{claim: X_3}
    The number of hyperedges containing at least one vertex from $X_3$ is at most $|X_3|\cdot({|A\setminus A'|\choose 2}+{|B\setminus B'|\choose 2}+{|C\setminus C'|\choose 2}) + |X_3|\cdot  6\delta^{\frac{3}{2}} n^2$.
\end{claim}

\begin{proof}

First, the number of hyperedges containing one vertex from $X_3$ and another vertex from $X$ is at most $|X_3| \cdot \frac{3}{2}\delta^2 n^2$.

Then we consider the hyperedges containing exactly one vertex from $X_3$ and two vertices from $V(\mathcal{H}) \setminus X$.

There are two types of hyperedges:
\begin{enumerate}
    \item The hyperedges containing one vertex from $X_3$ and two vertices from the same part.
    \item The hyperedges containing one vertex from $X_3$ and two vertices from different parts.
\end{enumerate}

For the first type, the number of such hyperedges is at most 

   $$ |X_3|\cdot\left( {|A\setminus A'|\choose 2}+{|B\setminus B'|\choose 2}+{|C\setminus C'|\choose 2}\right).$$ 

For the second type, by Claim~\ref{claim: X property}, the number of such hyperedges is at most $|X_3|\cdot 4\epsilon^{\frac{1}{2}} n^2=|X_3|\cdot 4\delta^{\frac{3}{2}}n^2$. By summing them up, we prove the claim.  
\end{proof}

Now we count the total number of hyperedges in $\mathcal{H}$.
Recall that $\epsilon=\delta^3$ and we set $\delta=\frac{1}{10t}$.
According to Claims~\ref{claim: empty V_i'},~\ref{claim: X1},~\ref{claim: X2} and~\ref{claim: X_3}, we have

\begin{equation}\label{eq: up e(H)}
\begin{aligned}
    e(\mathcal{H}) & \le |A\setminus A'||B\setminus B'||C\setminus C'| + |X_3|\cdot \left( {|A\setminus A'|\choose 2}+{|B\setminus B'|\choose 2}+{|C\setminus C'|\choose 2}\right) \\
    &+ |X_1|\cdot 5\delta^{\frac{3}{2}} n^2 + |X_2| \cdot\left( \frac{1}{9}n^2 + 3\epsilon n^2 - 2\delta n^2 \right) + |X_3|\cdot 6\delta^{\frac{3}{2}} n^2 \\ 
    & \le {\left( \frac{n - |X|}{3}\right)}^3 + 3|X_3|\binom{\frac{n - t}{3}}{2} + (|X_1|+|X_3|)\cdot 6\delta^{\frac{3}{2}} n^2 + |X_2| \left( \frac{1}{9}n^2  - \delta n^2 \right) \\
    & \le {\left( \frac{n}{3}\right)}^3-\frac{1}{9}\cdot n^2|X|+\frac{1}{9}\cdot n|X|^2 + 3|X_3| \binom{\frac{n - t}{3}}{2} \\& + (|X_1|+|X_3|)\cdot 6\delta^{\frac{3}{2}} n^2 + |X_2| \left( \frac{1}{9}n^2  - \delta n^2 \right)\\
    & \le {\left( \frac{n}{3}\right)}^3 + 3 |X_3| \binom{\frac{n - t}{3}}{2} - (|X_1|+|X_3|)(\frac{1}{9}-6\delta^{\frac{3}{2}})n^2  - |X_2| \delta n^2+\frac{1}{4}\delta^2 n^2|X|.   \\
\end{aligned}
\end{equation}


If $e(\mathcal{H}) \ge abc + t \left( \binom{a}{2} + \binom{b}{2}  + \binom{c}{2} \right)$ where $a + b + c = n-t$ and $a,b,c$ differ by at most $1$, then we have 
\begin{equation}\label{eq: low e(H)}
\begin{aligned}
    e(\mathcal{H}) & \ge \max_{a+b+c=n-t}\left\{abc+t\left( {a\choose 2}+{b\choose 2}+{c\choose 2}\right)\right\}+{t\choose 3}\\
    & \geq{\left( \frac{n - t}{3}\right)}^3 - o(n^2) + 3t\binom{\frac{n - t}{3}}{2} \\ 
    & \ge {\left( \frac{n}{3}\right)}^3 -t \frac{1}{9}n^2 + 3t \binom{\frac{n - t}{3}}{2} - o(n^2).
\end{aligned}
\end{equation}

Combining (\ref{eq: up e(H)}) and (\ref{eq: low e(H)}), we have $|X_3|\geq t$, and $$(|X_1|+|X_3|)\left(\frac{1}{9}-6\delta^{\frac{3}{2}}\right)+|X_2|\cdot\delta-\frac{1}{4}\delta^2|X|\leq \frac{1}{9}t.$$
Since $\delta=\frac{1}{10t}$, it implies $|X_1|=|X_2|=0$, and
$|X| = |X_3| = t$.

Consider now the case where there exists a hyperedge containing two vertices from $X$ and one from $(A\cup B\cup C)\setminus X$, say $v_1x_1x_2$ where $v_1 \in A\setminus A'$ and $x_1,x_2 \in X$. 

We claim that for a bad pair $y,z$ from $A\setminus A'$ or $B\setminus B'$ or $C\setminus C'$ with $v_1 \notin \{y,z\}$, one of $x_1yz$ and $x_2yz$ is not a hyperedge in $\cG$.
Otherwise, we can find a copy of $\cF$ in $\cG$ by letting $v_1$ play the role of $f_5$ in $\cF$ and using the same argument as above.
Then by Claim~\ref{claim: X property}, the number of hyperedges in $\cG$ is at most

\begin{equation*}
\begin{aligned}
    e(\mathcal{H}) & \le |A\setminus A'|\cdot|B\setminus B'|\cdot|C\setminus C'| + (t-1) \left({|A\setminus A'|\choose 2}+{|B\setminus B'|\choose 2}+{|C\setminus C'|\choose 2}\right)\\&+ t n + t\cdot 4\epsilon^{\frac{1}{2}} n^2,
\end{aligned}
\end{equation*}

a contradiction when $\epsilon=(\frac{1}{10t})^3$.

Consider now the case where there exists a hyperedge containing one vertex from $X$ and a proper pair in $V(\mathcal{G}) \setminus X$, say $x v_2 v_3$ where $x \in X, v_2 \in B\setminus B', v_3 \in C\setminus C'$.
By Claim~\ref{claim: X property}, for any bad pair $yz$ satisfying $xyz \in E(\cG)$, $y$ or $z$ must be in $\{v_2,v_3\}$.
Then the number of hyperedges in $\mathcal{G}$ is at most
\begin{equation*}
\begin{aligned}
    e(\mathcal{H}) & \le |A\setminus A'|\cdot|B\setminus B'|\cdot|C\setminus C'| + (t-1) \left({|A\setminus A'|\choose 2}+{|B\setminus B'|\choose 2}+{|C\setminus C'|\choose 2}\right)\\&+ 2n + t\cdot 4\epsilon^{\frac{1}{2}} n^2,
\end{aligned}
\end{equation*}
a contradiction when $\epsilon=(\frac{1}{10t})^3$.

Finally, consider the case where a hyperedge containing a vertex from $X$ is either inside $X$, or the other two vertices form an improper good pair in $V(\cG) \setminus X$.
Let $\mathcal{S}$ be the set of improper good pairs in $V(\cG) \setminus X$.
Then $|\mathcal{M}| \ge \left( \frac{1}{3}n - \epsilon^{\frac{1}{2}} n \right) |\mathcal{S}|/3$.

Then the number of hyperedges in $\mathcal{G}$ is at most
\begin{equation*}
\begin{aligned}
    e(\mathcal{H}) & \le |A\setminus A'|\cdot|B\setminus B'|\cdot|C\setminus C'| - |\mathcal{M}| \\
    &+ t \left({|A\setminus A'|\choose 2}+{|B\setminus B'|\choose 2}+{|C\setminus C'|\choose 2}\right) + t |\mathcal{S}| + \binom{t}{3} \\ 
    & \le|A\setminus A'|\cdot|B\setminus B'|\cdot|C\setminus C'|+ t \left({|A\setminus A'|\choose 2}+{|B\setminus B'|\choose 2}+{|C\setminus C'|\choose 2}\right) \\
    &+ \binom{t}{3} - \left( \frac{1}{9}n - \frac{1}{3}\epsilon^{\frac{1}{2}} n - t \right) |\mathcal{S}|\\
    &\leq \max_{a+b+c=n-t}\left\{abc+t\left( {a\choose 2}+{b\choose 2}+{c\choose 2}\right)\right\}
    +{t\choose 3}-|\mathcal{S}|\cdot\left(\frac{1}{9}n-\frac{1}{3}\epsilon^{\frac{1}{2}}n-t\right).
\end{aligned}
\end{equation*}

Comparing this with the claimed bound, we have $|\mathcal{S}| = 0$ when $\epsilon=(\frac{1}{10t})^3$.

That is, $A\setminus A', B\setminus B', C\setminus C'$ form a complete tripartite $3$-uniform hypergraph in $\cG$, 
$X = \{x_1,\ldots,x_t\}$ and for every pair $yz$ in the same part, $yzx_\ell$ is a hyperedge in $\cG$ for each $\ell=1,2,\ldots,t$.
Moreover, $X$ is a clique. This completes the proof.
\end{proof}

\section{Blowing up one vertex in \texorpdfstring{$\cF_5$}{F\_5}}\label{sec: section 5}
In this section, we study the Tur\'an number of the hypergraph $\cF_5(f_3;t)$. We rewrite Theorem \ref{thm: F5 v3 t} here for convenience.

\medskip

\noindent\textbf{Theorem.}
{\it
    For sufficiently large $n$, we have $\ex_3(n,\cF_5(f_3;t))=e(T_3(n,3)^{t+})$.
}

\smallskip

\begin{proof}
    The lower bound is given by $T_3(n,3)^{t+}$ that we defined in the Introduction.
    Notice that $\cF_5(f_3;t)$ is a subgraph of a blow-up of $\cF_5[t]$.
    Let $\cG$ be the extremal hypergraph for $\cF_5(f_3;t)$. For $\epsilon^{\frac{1}{3}}=\delta < \frac{1}{3}(\frac{1}{18})^3$, let $A, B, C$ be a partition of $V(\cG)$ obtained from Theorem \ref{thm: blow-up sta}. The definitions of $A',B',C'$ and $\cM $, $\cE $ are the same as in Section \ref{sec: notation}.
    We aim to prove that 
    \begin{equation}\label{eq: goal in F5f3t}
        |\cE |<\frac{1}{2}|\cM |+S^t(|A|)+S^t(|B|)+S^t(|C|),
    \end{equation}
     where $S^t(x)$ denotes the maximum size of a $3$-uniform hypergraph on $x$ vertices in which every two vertices are contained in at most $t-1$ hyperedges. If (\ref{eq: goal in F5f3t}) holds, then the number of hyperedges in $\cG$ is at most 
    $$|E(\cK)|-|\cM| +|\cE| \leq e(T_3(n,3))+S^t(|A|)+S^t(|B|)+S^t(|C|).$$
    We define the following sets of hyperedges.
    Let $\cE_1 $ denote the collection of hyperedges in $\cE $ that contain exactly two vertices from one part and one vertex from another part. We divide $\cE_1 $ into the following parts.
    $$\cE_1^{A,B} =\{E\in \cE :|E\cap A|=2,|E\cap B|=1\},$$
    and we similarly define $\cE_1^{A,C} , \cE_1^{B,A} , \cE_1^{B,C} , \cE_1^{C,A} ,$ and $\cE_1^{C,B} $. 
    We define $\cE_2 $ as the collection of hyperedges in $\cE $ that are contained in one part, and we divide $\cE_2 $ into the following parts.
    $$\cE_2^{A} =\{E\in \cE :E\subseteq A\},$$
    and we similarly define $\cE_2^{B} , \cE_2^{C} $. For $t$ hyperedges $E_1,\dots, E_t$, we call them a {\it $t$-sunflower} with core $\{u,v\}$ if $E_i\cap E_j=\{u,v\}$ for every $1\leq i<j\leq t$. Set
    $$ \cE_2^{A,t}=\{E=uvw\in \cE_2^A : \text{$E$ is in a $t$-sunflower with hyperedges in $\cE_2^A $}\},$$
    and we similarly define $\cE_2^{B,t}$ and $\cE_2^{C,t}$.
    
    First, we deal with $ \cE_2^{A,t}$. 
    If $ \cE_2^{A,t}=\emptyset$, then we have $|\cE_2^A |\leq S^t(|A|)$.
    Otherwise, 
    suppose there are $E_1,\dots,E_t\in \cE_2^{A,t} $ such that $E_i\cap E_j=\{u,v\}$ for every $1\leq i<j\leq t$. Let $a_i=E_i\setminus \{u,v\}$ for every $1\leq i\leq t$.
    \begin{claim}\label{claim: sunflower intersect A'}
    If $E_i=\{u,v,a_i\}$, $i\in[t]$ form a $t$-sunflower, then
        at least one of $\{a_1,\dots,a_t,u\}$ is contained in $A'$, and at least one of $\{a_1,\dots,a_t,v\}$ is contained in $A'$.
    \end{claim}
    \begin{proof}
        Suppose to the contrary that $\{a_1,\dots,a_t,u\}\cap A'=\emptyset$, then $\{a_1,\dots,a_t,u\}\subseteq A\setminus A'$.
        Note that for every $v\in A\setminus A'$, $$d_{\cM}(v)\leq 2\delta n^2.$$
        Let us consider the intersection of the link graph of $a_1,\dots,a_t$ and $u$. We have 
        $$\left|\left(\bigcap_{i=1}^t L(\cG_g(a_i))\right)\bigcap L(\cG_g(u))\right|\geq |B|\cdot |C|-2\delta (t+1) n^2>1$$
        when $\delta$ is sufficiently small. Then there exists a good pair $bc$ with $b\in B,c\in C$ such that $a_ibc,u bc\in \cG$ for every $1\leq i\leq t$. They form a copy of $\cF_5(f_3;t)$ in $\cG$, a contradiction (see Figure
        \ref{fig:find a F5f3t}.)
        \begin{figure}
            \centering
        \includegraphics[width=0.5\linewidth]{ 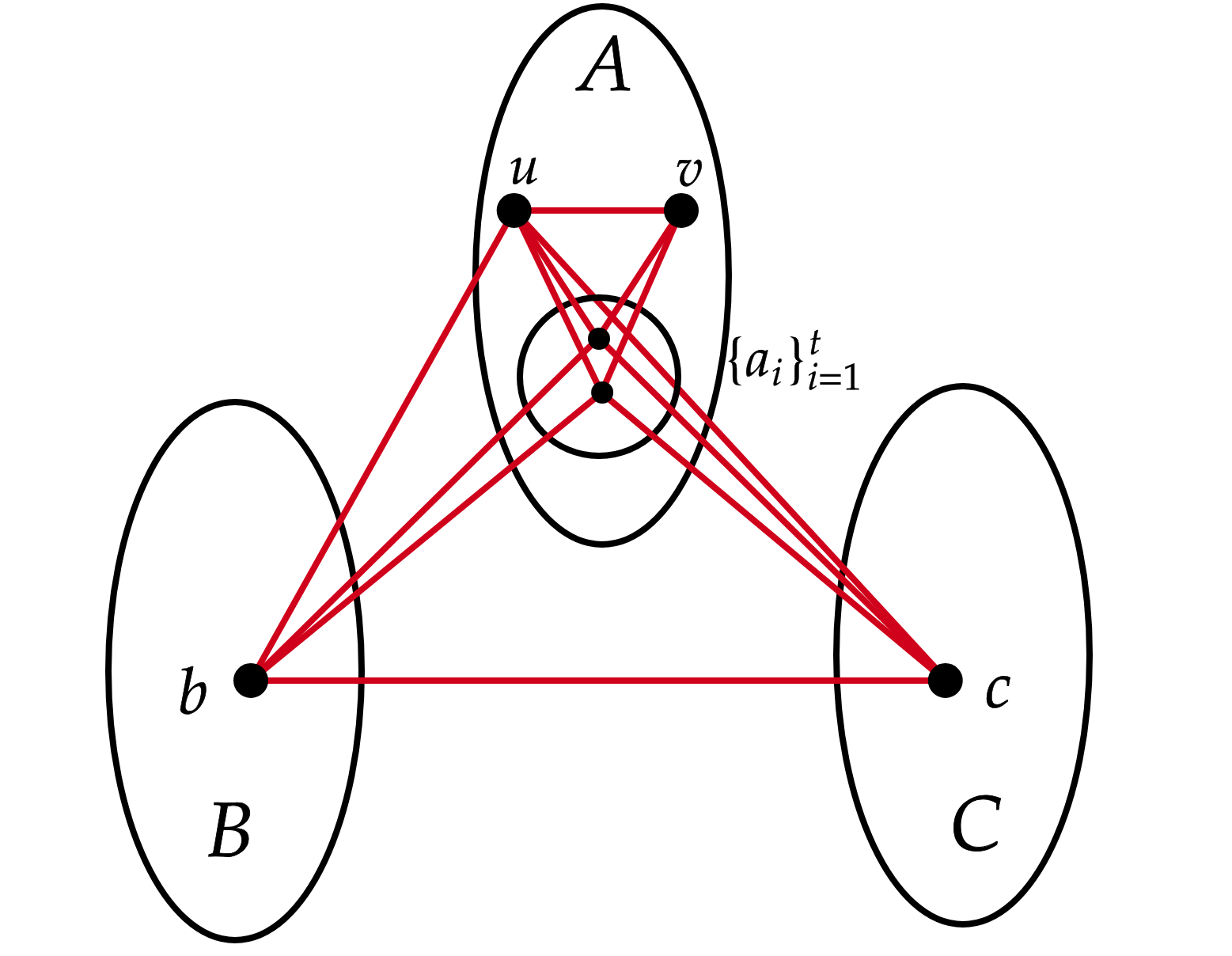}
            \caption{The vertices $\{u,v,b,c\}$ and $\{a_i\}_{i=1}^t$ forms a copy of $\cF_5(f_3;t)$}
            \label{fig:find a F5f3t}
        \end{figure}
        Symmetrically, we can prove that $\{a_1,\dots,a_t,v\}\cap A'\neq \emptyset$.
    \end{proof}
    \begin{claim}\label{eq: upper bound of E2 in F5f3t}
        We have 
        $$|\cE_2 |\leq 6\delta \cdot|\cM |+|S^t(|A|)|+S^t(|B|)+S^t(|C|).$$
    \end{claim}
    \begin{proof}
    We only give an upper bound on the size of $\cE_2^A $, since the other two cases are similar.
    Let $\cE_2^{A'_{\geq 2}} =\{E\in \cE_2^A : |E\cap A'|\geq 2\}$.
    For every $u\in A'$, note that $d_{\cM}(u)\geq 2\delta n^2$, and the number of hyperedges in $\cE_2^{A'_{\geq 2}} $ that contain $u$ is at most $|A'|\cdot n\leq \frac{1}{2}\delta^2n^2$. We have
    \begin{equation}\label{eq: E2 A' geq 2}
        \frac{|\cE_2^{A'_{\geq 2}} |}{|\cM |}\leq \frac{\frac{1}{2}\delta^2n^2}{2\delta n^2}=\frac{1}{4}\delta.
    \end{equation}
    Next, we consider the hyperedges in $\cE_2^A $ that contain exactly one vertex in $A'$, denoted by $\cE_2^{A'_{=1}} $. 
    We focus on the link graph of a fixed $u\in A'$ with hyperedges in $\cE_2^{A'_{=1}} $, which is denoted by $L(\cE_2^{A'_{=1}}(u))$, and shortened by $L_u$. 
    Let $e(L_u)=|\cE_2^{A'_{=1}}(u)|$. We claim that there is no vertex in $L_u$ with degree at least $t$. Otherwise, there exists a $t$-sunflower containing $u$ such that the other $t+1$ vertices are in $A\setminus A'$, a contradiction with Claim \ref{claim: sunflower intersect A'}.
    As a result, $e(L_u)\leq (t-1)n$. Then we have
    \begin{equation}\label{eq: E2 A' =1}
        \frac{|\cE_2^{A'_{=1}}|}{|\cM |}\leq \frac{(t-1)n}{2\delta n^2}<\delta.
    \end{equation}
    According to Claim \ref{claim: sunflower intersect A'}, there is no $t$-sunflower with hyperedges in $\cE_2^A \setminus((\cE_2^{A'_{=1}} )\cup\cE_2^{A'_{\geq 2}} )$. Hence, we have
    $$|\cE_2^A \setminus((\cE_2^{A'_{=1}} )\cup\cE_2^{A'_{\geq 2}} )|\leq S^t(|A|).$$
    Together with (\ref{eq: E2 A' geq 2}) and (\ref{eq: E2 A' =1}), we have $$|\cE_2^A |\leq 2\delta\cdot |\cM |+S^t(|A|).$$ By symmetry, the same bound holds for $|\cE_2^B |$ and $|\cE_2^C|$, thus the claim holds.
    \end{proof}
Next, we deal with $\cE_1^{A,B} $.
We divide $\cE_1^{A,B} $ into three parts.
$$\cE_{1,1}^{A,B} =\{E\in \cE_1^{A,B} :|E\cap A'|=2,\text{~and~} |E\cap B|=1\}.$$
$$\cE_{1,2}^{A,B} =\{E\in \cE_1^{A,B} :|E\cap A'|=1, |E\cap (A\setminus A')|=1,\text{~and~} |E\cap B|=1\}.$$
$$\cE_{1,3}^{A,B} =\{E\in \cE_1^{A,B} :|E\cap (A\setminus A')|=2,\text{~and~} |E\cap B|=1\}.$$
Recall that $|A'|\leq \frac{1}{2}\delta ^2 n,$ and for every $u\in A'$, the number of hyperedges in $\cE_{1,1}^{A,B} $ that contain $u$, which is denoted by $|\cE_{1,1}^{A,B} (u)|=d_{\cE_{1,1}^{A,B}}(u)$, is at most $|A'|\cdot n\leq \frac{1}{2}\delta^2 n^2$.
For every $u\in A'$, according to the definition of $A'$, we have $d_{\cM }(u)\geq 2\delta n^2$,
thus we have 
 \begin{equation}\label{eq: E1 1 A B}
    \frac{|\cE_{1,1}^{A,B} |}{|\cM|}\leq \frac{\frac{1}{2}\delta^2 n^2}{2\delta n^2}=\frac{1}{4}\delta.
 \end{equation}
 Next, we deal with $\cE_{1,2}^{A,B} $.
 \begin{claim}\label{claim: E1 2 A B}
    We have 
    $$|\cE_{1,2}^{A,B} |\leq \delta^{\frac{1}{3}} \cdot |\cM |.$$
 \end{claim}
 \begin{proof}
    Let $\cE_{1,2}^{A,B} (u)$ denote the hyperedges in $\cE_{1,2}^{A,B} $ that contain $u\in A'$. 
 If for every $u\in A'$ we have $|\cE_{1,2}^{A,B} (u)|\leq \delta^{\frac{4}{3}} n^2$,
   then we have
    $$\frac{|\cE_{1,2}^{A,B} |}{|\cM |}\leq \frac{\delta^{\frac{4}{3}} n^2}{2\delta n^2}<\delta^{\frac{1}{3}},$$
    and we are done.

    Therefore, we may suppose that for some $u\in A'$, $|\cE_{1,2}^{A,B} (u)|\geq \delta^{\frac{4}{3}} n^2$.
    Recall that the improper pairs in $B\times C$ are the pairs $(b,c)$ such that $d_{\cM }(\{b,c\})\geq \epsilon^{\frac{1}{2}} n$.
     By Section \ref{sec: notation}, the number of improper pairs in $A\times B$ is at most $3\epsilon^{1/2} n^2<\epsilon^{4/9}n^2=\delta^{\frac{4}{3}} n^2$.
    Thus, there exists a proper pair $(a_0,b_0)\in A\times B$ such that $a_0ub_0\in \cE_{1,2}^{A,B} (u)$, which implies $d_{\cM }(\{a_0,b_0\})\leq \epsilon^{1/2} n$.
    Since we supposed that $\cG$ has the maximum number of good hyperedges, we have $d_{\cG_g}(u)\geq d_{\cE_{1,2}^{A,B} }(u)\geq \delta^{\frac{4}{3}} n^2$.
    Consider the link graph $L(\cG_g(u))$, which is a bipartite graph with parts $B$ and $C$ and has at least $\delta^{\frac{4}{3}} n^2$ edges.
    Then, there exists $b_1\in B$ such that $$d_{L(\cG_g(u))}(b_1)\geq \frac{\delta^{\frac{4}{3}} n^2}{|B|}\geq 2\delta^{\frac{4}{3}} n.$$
    We have $d_{L(\cG_g(u))}(b_1)-d_{\cM }(\{a_0,b_0\})\geq 2\delta^{\frac{4}{3}} n-\epsilon^{1/2} n>t$, when $\epsilon$ is sufficiently small.
    Thus, there exist $t$ vertices $\{c_1,\dots,c_t\}\subseteq C$ such that $ub_1c_i,a_0b_0c_i\in E(\cG)$ for every $1\leq i\leq t$.
     They form a copy of $\cF_5(f_3;t)$ in $\cG$, a contradiction.
 \end{proof}
Finally, we deal with $\cE_{1,3}^{A,B} $. For every $E\in \cE_{1,3}^{A,B} $, suppose $E=a_1a_2b$ with $a_1,a_2\in A\setminus A'$ and $b\in B$.
We divide hyperedges $E$ in $\cE_{1,3}^{A,B} $ into two parts according to the degree of $b$.
$$\cF_1=\{E=a_1a_2b\in \cE_{1,3}^{A,B} : |\cE_{1,3}^{A,B} (b)|\leq \delta n^2\},$$
$$\cF_2=\{E=a_1a_2b\in \cE_{1,3}^{A,B} : |\cE_{1,3}^{A,B} (b)|>\delta n^2\}.$$
\begin{claim}\label{claim: bound of F1}
    We have 
    $$|\cF_1|\leq 2\delta^{\frac{1}{2}} \cdot |\cM |.$$
\end{claim}
\begin{proof}
     Since $a_1\in A\setminus A'$, we have $d_{\cG_g}(a_1)\geq |B|\cdot |C|-2\delta n^2\geq (\frac{1}{3}-\epsilon)^2n^2\geq (\frac{1}{3}n)^2-3\delta n^2$.
    Then, by considering the link graph $L(\cG_g(a_1))$, there exists $b_1\in B$ such that $d_{L(\cG_g(a_1))}(b_1)\geq \frac{1}{3}n-12\delta n$.
    If $d_{\cM }(\{a_2,b\})\leq \frac{1}{3}n-12\delta n-t$, then there exist $t$ vertices $\{c_1,\dots,c_t\}\subseteq C$ such that $a_1b_1c_i,a_2bc_i\in E(\cG)$ for every $1\leq i\leq t$.
    They form a copy of $\cF_5(f_3;t)$ in $\cG$, a contradiction.
    Thus, we have $d_{\cM }(\{a_2,b\})>\frac{1}{3}n-12\delta n-t\geq \frac{1}{3}n-13\delta n$.
    By symmetry, we also have $d_{\cM }(\{a_1,b\})>\frac{1}{3}n-13\delta n$.
    Let $$Y_b=\bigcup_{E\in \cE_{1,3}^{A,B} (b)}E\cap (A\setminus A').$$
    Then, $Y_b$ is the vertex set of the link graph of $b$ in the hypergraph $\cE_{1,3}^{A,B} (b)$, hence
    the size of $Y_b$ is at least $\sqrt{2|\cE_{1,3}^{A,B} (b)|}$, 
    and for every $y\in Y_b$, we have $d_{\cM }(\{y,b\})>\frac{1}{3}n-13\delta n$.
    So we have
    \begin{equation}\label{eq: dM(b) for b in E13AB}
       d_{\cM }(b)\geq |Y_b|\left(\frac{1}{3}n-13\delta n\right)\geq \sqrt{2|\cE_{1,3}^{A,B} (b)|}\left(\frac{1}{3}n-13\delta n\right). 
    \end{equation} 
    Since $|\cE_{1,3}^{A,B}(b)|\leq \delta n^2$, we have for every $b\in B$ with $|\cE_{1,3}^{A,B} (b)|\leq \delta n^2$,
    $$\frac{|\cF_1|}{|\cM |}\leq \frac{\sqrt{|\cE_{1,3}^{A,B} (b)|}}{\sqrt{2}\left(\frac{1}{3}n-13\delta n\right)}\leq \frac{\delta^{\frac{1}{2}}}{\sqrt{2}\left(\frac{1}{3}-13\delta\right)}<2\delta^{\frac{1}{2}}.$$
\end{proof}
  \begin{claim}\label{claim: size of F2}
    We have 
    $$\cF_2=\emptyset.$$
  \end{claim}
\begin{proof}
    For every $b\in B$ with $|\cE_{1,3}^{A,B} (b)|>\delta n^2$, we consider the link graph $L(\cE_{1,3}^{A,B} (b))$, which is a graph with vertex set $A\setminus A'$.
    Then, since $e(L(\cE_{1,3}^{A,B} (b)))>\delta n^2$, there exists a vertex $a_0\in A\setminus A'$ such that $d_{L(\cE_{1,3}^{A,B} (b))}(a_0)\geq t$, which implies that there exist $t$ vertices $\{a_1,\dots,a_t\}\subseteq A\setminus A'$ such that $a_0a_ib\in \cE_{1,3}^{A,B} (b)$ for every $1\leq i\leq t$.

    Since $\{a_0,a_1,\dots,a_t\}\subseteq A\setminus A'$, we can similarly find a pair $b_0c_0$ with $b_0\in B,c_0\in C$ such that $a_0b_0c_0\in E(\cG)$ and $a_ib_0c_0\in E(\cG)$ for every $1\leq i\leq t$.
    They form a copy of $\cF_5(f_3;t)$ in $\cG$, a contradiction.
    Thus, $\cF_2=\emptyset$.
\end{proof}
Combining Claim \ref{claim: bound of F1} and Claim \ref{claim: size of F2}, we have
\begin{equation}\label{eq: E1 3 A B}
   |\cE_{1,3}^{A,B} |\leq 2\delta^{\frac{1}{2}}\cdot |\cM | .
\end{equation}
Combining (\ref{eq: E1 1 A B}), Claim \ref{claim: E1 2 A B} and (\ref{eq: E1 3 A B}), we have
$$|\cE_1^{A,B} |\leq 3\delta^{\frac{1}{3}}\cdot |\cM |.$$
By symmetric argument, we have all of $\cE_1^{A,C} , \cE_1^{B,A} , \cE_1^{B,C} ,\cE_1^{C,A} ,\cE_1^{C,B} $ satisfy the same bound.
As a result, we have 
$$|\cE_1 |\leq 18\delta^{\frac{1}{3}}|\cM |.$$
Together with Claim \ref{eq: upper bound of E2 in F5f3t}, the inequality (\ref{eq: goal in F5f3t}) holds when $\delta<\frac{1}{3}(\frac{1}{18})^3$. This completes the proof.
\end{proof}

\section{Tur{\'a}n number of \texorpdfstring{$(t+1)\cdot \cF_5^S(m)$}{t F\_5^S(m)} }\label{sec: section 6}
In this section, we study the Tur\'an number of the hypergraph 
that consists of $t+1$ vertex-disjoint copies of $\cF_5^S(m)$.
We rewrite Theorem \ref{thm: t cF_5^S(m)} here for convenience.

\noindent\textbf{Theorem.}
{\it
    For every $m\geq 2$, when $n$ is sufficiently large, the following hypergraph $\cG(t,m,n)$ is the extremal hypergraph of $(t+1)\cdot \cF_5^S(m)$ with maximum number of hyperedges.

    Fix a set $T$ with $t$ vertices, and a hypergraph $T_3(n-t,3)$ disjoint with $T$.
Let $\cG(t,m,n)$ denote the $3$-uniform $n$-vertex hypergraph obtained by adding all the hyperedges intersecting $T$ to the union of $T$ and $T_3(n-t,3)$. 

}

\smallskip

\begin{proof}
    Let $\cG$ be the extremal $3$-uniform hypergraph for $(t+1)\cdot \cF_5^S (m)$. For $\epsilon >0$, let $A,B,C$ be a partition of $V(\cG)$ obtained from Theorem \ref{thm: blow-up sta}.
    The definition of $A',B',C'$ and $\cM $, $\cE $ are the same as in Section \ref{sec: notation}.
    Our aim is to prove that all but at most $4t\delta^{\frac{1}{2}}|\cM |$ hyperedges in $\cE $ intersect with $t$ fixed vertices. Then we have that the upper bound on the number of hyperedges in $\cG$ is at most the number of hyperedges in a three-partite complete $3$-uniform hypergraph plus the number of extra hyperedges intersecting with $t$ fixed vertices.
    Finally, we will prove $\cG(t,m,n)$ achieves the extremal number. 

    We call a hyperedge $E\in \cE$ \textbf{based} if $|E\cap (A\setminus A')|\geq 2$, or $|E\cap (B\setminus B')|\geq 2$, or $|E\cap (C\setminus C')|\geq 2$. Other hyperedges are called \textbf{unbased}. We set $X=A'\cup B'\cup C'$.
    Then, there are two types of unbased hyperedges in $\cE $ that contain at least two vertices in $A$.
    $$\cE_1 =\{E=a_1a_2x\in \cE : a_1,a_2\in A,|E\cap X|\geq 2\}.$$
    $$\cE_2 =\{E=a_1a_2x\in \cE : a_1\in A',a_2\in A\setminus A',x\in (B\cup C)\setminus(B'\cup C')\}.$$
    According to the definition of proper pairs, we divide $\cE_2 $ into two parts.
    $$\cE_2^1 =\{E=a_1a_2x\in \cE_2 : a_1\in A', (a_2,x) \text{~is a proper pair}\},$$
    $$\cE_2^2 =\{E=a_1a_2x\in \cE_2 : a_1\in A', (a_2,x) \text{~is an improper pair}\}.$$
    \begin{clm}\label{clm: upper bound of E22}
        When $n$ is sufficiently large, we have
        $|\cE_1 |\leq \delta |\cM |$, and $|\cE_2^2 |\leq 2\delta^{\frac{1}{2}} |\cM |$.
    \end{clm}
    \begin{proof}
        Let $E=a_1a_2x\in \cE_1 $ with $a_1,a_2\in A$. Since $|E\cap X|\geq 2$, we may assume $a_1\in A'$, and then $d_{\cM }(a_1)\geq 2\delta n^2$.
        We also have $d_{\cE_1 }(a_1)\leq |A'|n\leq \frac{1}{2}\delta^2 n^2$.
        Thus, we have
        $$\frac{|\cE_1 |}{|\cM |}\leq \frac{\frac{1}{2}\delta^2 n^2}{2\delta n^2}=\frac{1}{4}\delta<\delta.$$
        
        Then, we deal with $\cE_2^2 $. 
        Since the number of improper pairs in $B\times C$ is at most $3\epsilon^{\frac{1}{2}}n^2$, we have $|\cE_2^2 (u)|\leq 3\epsilon^{\frac{1}{2}} n^2$, and by the definition of $A'$, $|d_{\cM }(u)|\geq 2\delta n^2$. 
        Thus,
        $$\frac{|\cE_2^2 |}{|\cM |}\leq \frac{3\epsilon^{\frac{1}{2}} n^2}{2\delta n^2}<2\delta^{\frac{1}{2}}.$$
\end{proof}
    \begin{claim}\label{claim: extend based edge to Sm}
        For every integer $k>0$, every $E\in \cE $ which is a based hyperedge, and a set of vertices $S\subseteq V(\cG)$ with size $k$ and disjoint with $E$, when $n$ is large enough, there exists a copy of $\cF_5^S(m)$ in $\cG$ that contains $E$ and disjoint with $S$.
    \end{claim}
    \begin{proof}
        According to the definition of the based hyperedges, we may assume $E=a_1a_2x$ where $a_1,a_2\in A\setminus A'$.
        Then, we choose $m-2$ vertices $\{a_3,\dots,a_m\}\in A\setminus (A'\cup S\cup \{a_1,a_2,x\})$.
        Notice that $a_i\in A\setminus A'$, we have $d_{\cM }(a_i)\leq 2\delta n^2$. It implies that there are at least $|B|\cdot |C|-2m\delta n^2-kn$ pairs $(b,c)\in (B\setminus S)\times (C\setminus S)$ such that $a_ibc\in E(\cG)$ for every $1\leq i\leq m$.
        These pairs form a copy of $K_{m,m}$ contained in $(B\setminus S)\times(C\setminus S)$.
        Thus, there exists a copy of $\cF_5^S(m)$ in $\cG$ that contains $E$ and is disjoint from $S$.
    \end{proof}
       Suppose that the maximum number of disjoint copies of $\cF_5^S(m)$ in $\cG$ is $t'\leq t$. 
       For every set $S$ of vertices, let
       $$\cE_2^1(\cG-S)=\{E=a_1a_2x\in \cE_2^1 : E\cap S=\emptyset, \text{~ $E$ is not contained in $\cF_5^S(m)$ of $\cG-S$}\}.$$
       Then we have the following result.
       \begin{claim}\label{claim: bound of E2 1 in G-S}
        For every integer $k>0$, and a set of vertices $S\subseteq V(\cG)$ with size $k$, when $n$ is sufficiently large, we have
        $$|\cE_2^1(\cG-S)|\leq 10(2m+1)\delta |\cM |<\delta^{\frac{1}{2}}|\cM|.$$
       \end{claim}
       \begin{proof}
       Fix $a_1\in A'\setminus S$, and $a_1a_2x\in \cE_2^1(\cG-S)$. Choose $m-2$ other vertices $a_3,\dots,a_m\in A\setminus (A'\cup S\cup \{a_1,a_2,x\})$. We construct an auxiliary bipartite graph $B_{a_1}$ with parts $B\setminus S$ and $C\setminus S$, where $(b,c)\in E(B_{a_1})$ if and only if $a_ibc\in E(\cG)$ for every $2\leq i\leq m$.
       Recall that a hyperedge $E\in E(\cG)$ is good if it intersects each part with exactly one vertex, and $\cG_g$ is the collection of all the good hyperedges.
       Let $L(\cG_g-S)(a_1)$ denote the link graph of $a_1$ in the hypergraph that consists of the good hyperedges that are disjoint with $S$. 
       Then $L(\cG_g-S)(a_1)$ is a bipartite graph with parts $B\setminus S$ and $C\setminus S$.
       Note that if there exists a copy of $K_{m,m}$ in $E(B_{a_1})\cap E(L(\cG_g-S)(a_1))$, then we can find a copy of $\cF_5^S(m)$ in $\cG$ that contains $a_1a_2x$ and is disjoint with $S$,
      contradicting the definition of $\cE_2^1(\cG-S)$.
        As a result, $E(B_{a_1})\cap E(L(\cG_g-S)(a_1))$ is $K_{m,m}$-free. Then by Theorem \ref{thm: ex of Kss}, we have
        $$|E(B_{a_1})\cap E(L(\cG_g-S)(a_1))|\leq \delta n^2.$$
 Similarly to Claim \ref{claim: extend based edge to Sm}, we have $e(B_{a_1})\geq |B|\cdot |C|-2m\delta n^2$.
        It implies that $|\cG_g(a_1)|\leq (2m+1)\delta n^2$, and thus the number of hyperedges in $\cM$ containing $a_1$ is at least 
        $  |B|\cdot |C|-(2m+1)\delta n^2.$

        
        We suppose that $\cG$ has the maximum number of good hyperedges, then we have $|\cE_2^1(\cG-S)(a_1)|\leq (2m+1)\delta n^2$.
        Thus,
        $$\frac{|\cE_2^1(\cG-S)|}{|\cM |}\leq \frac{(2m+1)\delta n^2}{|B|\cdot |C|-(2m+1)\delta n^2}<10(2m+1)\delta.$$
        \end{proof}
        Suppose $H_1,\dots,H_{t'}$ are the $t'$ vertex-disjoint copies of $\cF_5^S(m)$ in $\cG$.
        Let $S_0=\bigcup_{i=1}^{t'} V(H_i)$. By Claim \ref{claim: bound of E2 1 in G-S}, when $n$ is sufficiently large, we have
        $$|\cE_2^1(\cG-S_0)|\leq 10(2m+1)\delta |\cM |.$$
        Every based hyperedge 
intersects with $S_0$, otherwise, by Claim \ref{claim: extend based edge to Sm}, we can find another copy of $\cF_5^S(m)$ in $\cG$ disjoint with $S_0$. We let $\cE_2^2(\cG,S_0)$ be the hyperedges in $\cE_2^2 $ that intersect with $S_0$.  Let $\cE_b $ be the collection of based hyperedges in $\cE $.
        Thus, the hyperedges in $\cE $ that do not intersect with $S_0$ are contained in $\cE_1 \cup \cE_2^2 \cup \cE_2^1(\cG-S_0)$, which has size at most $$(\delta+2\delta^{\frac{1}{2}}+10(2m+1)\delta)|\cM |<3\delta^{\frac{1}{2}}|\cM |.$$
        
        By the lower bound of $e(\cG)$, we have
        $$e(\cG) \geq  e(\cG(t,m,n))\geq e(T_3(n,3))+3t{\lf n/3\rf-t \choose 2}.$$
         We also have
        \begin{align*}
            e(\cG)=&e(\cK)-|\cM |+|\cE |\\
            =&e(\cK)-|\cM |+|\cE_1 \cup \cE_2^1 \cup \cE_2^2(\cG-S_0)|+|\cE_b |+|\cE_2^2(\cG,S_0)|.
        \end{align*}
        Hence, the size of $\cE_b \cup \cE_2^2(\cG,S_0)$ is at least.
        $$3t\binom{\lfloor n/3\rfloor-t}{2}+(1-3\delta^{\frac{1}{2}})|\cM |.$$
        The number of hyperedges in $\cG$ that intersect with two vertices in $S_0$ is at most $\binom{(3m+1)t'}{2}n\leq \binom{(3m+1)t}{2}n$. 
         Hence, the number of hyperedges in $\cE $ that intersect exactly one vertex in $S_0$ is at least
        $$3t\binom{\lfloor n/3\rfloor-t}{2}+(1-3\delta^{\frac{1}{2}})|\cM |-\binom{(3m+1)t}{2}n-\binom{(3m+1)t}{3}.$$
        We set
        $$\cF_i=\{E\in \cE_b \cup \cE_2^2(\cG,S_0): |E\cap H_i|=|E\cap S_0|=1\}.$$
        We have 
        \begin{equation}\label{eq: sum of Fi}
            \sum_{i=1}^{t'}|\cF_i|\geq 3t\binom{\lfloor n/3\rfloor-t}{2}+(1-3\delta^{\frac{1}{2}})|\cM |-\binom{(3m+1)t}{2}n-\binom{(3m+1)t}{3}.
        \end{equation}
        For every copy $H_i$, if there are two based hyperedges $F_1,F_2\in \cF_i$ and $F_1\cap F_2=\emptyset$, then by Claim \ref{claim: extend based edge to Sm}, we can find a copy of $2\cF_5^S(m)$ in $\cG$ that contains $F_1,F_2$ and avoids other $H_j$, $j\neq i$. In this way, we find a copy of $(t'+1)\cdot \cF_5^S(m)$ in $\cG$, a contradiction. 
        
        We first consider the case where $X=A'\cup B'\cup C'\neq \emptyset$, and then $|\cM |\geq \delta n^2$. When $X=\emptyset$, we can similarly deal with the problem by considering only based hyperedges.
        \begin{claim}\label{claim: bound of Fi}
            We have $|\cF_i|\leq 3\frac{(n/3)^2}{2}+(3m+1)\delta^{\frac{1}{2}}|\cM |$. Moreover, when $|\cF_i| \geq (3m+1)\delta^{\frac{1}{2}}|\cM |$, all but at most $2\delta^{\frac{1}{2}}|\cM |$ hyperedges in $\cF_i$ intersect $H_i$ with one fixed vertex. 
        \end{claim}
        \begin{proof}
 We prove this by considering the following two cases.

        \textbf{Case 1:} There is one based hyperedge $F_1$ in $\mathcal{F}_i$ that intersects $H_i$ with $u_1$.
        
        Then the other based hyperedges in $\mathcal{F}_i$ that intersects $H_i$ in a vertex of $V(H_i)\setminus \{u_1\}$ must intersect with $F_1$, thus the number of such hyperedges is at most $6mn$.
        
        There is no based hyperedge in $\mathcal{F}_i$ not intersecting with $F_1$.
        The number of hyperedges in $\cE_2^2(\cG,S_0)$ not intersecting with $F_1$ is at most 
        $\delta^{\frac{1}{2}}|\cM |$, otherwise, by Claim \ref{claim: extend based edge to Sm} and Claim \ref{claim: bound of E2 1 in G-S}, we can find another copy of $\cF_5^S(m)$ in $\cG$ that is disjoint with $S_0$ and $F_1$. 
        Then there exists $t'+1$ vertex-disjoint copies of $\cF_5^S(m)$ in $\cG$, a contradiction.
        It implies that all but at most $\delta^{\frac{1}{2}}|\cM |+6mt'n<2\delta^{\frac{1}{2}}|\cM |$ hyperedges in $\mathcal{F}_i$ intersect $H_i$ with $u_1$. Since the number of hyperedges in $\mathcal{F}_i$ containing $u_1$ is at most $3\frac{(n/3)^2}{2}$, we are done.
        
        \textbf{Case 2:} There are no based hyperedges in $\mathcal{F}_i$.

        Then all hyperedges in $\mathcal{F}_i$ are contained in $\cE_2^2(\cG,S_0)$.
        Then, the upper bound of $|\cF_i|$ is given by Claim \ref{clm: upper bound of E22}.

        If $|\cF_i|\geq \delta^{\frac{1}{2}}|\cM |$, then by the pigeonhole principle, we may assume that there are at least $\frac{1}{3m+1}\cdot(3m+1)\delta^{\frac{1}{2}}|\cM |>\delta^{\frac{1}{2}}|\cM |$ hyperedges in $\mathcal{F}_i$ intersecting $H_i$ with $u_1$.
        Then, with a similar proof as in Claim \ref{claim: bound of E2 1 in G-S}, there exists a copy of $\cF_5^S(m)$ in $\cG$ disjoint with the other $t'-1$ copies of $\cF_5^S(m)$ and containing one hyperedge intersecting $H_i$ with $u_1$. 
        And if there are at least $\delta^{\frac{1}{2}}|\cM |+6mt'n$ hyperedges in $\mathcal{F}_i$ intersecting $H_i$ with $V(H_i)
        \setminus \{u_1\}$, then by Claim \ref{claim: extend based edge to Sm} and Claim \ref{claim: bound of E2 1 in G-S}, there exists another copy of $\cF_5^S(m)$ in $\cG$ disjoint with the previous $t'$ copies of $\cF_5^S(m)$, a contradiction. 
        Thus, all but at most $\delta^{\frac{1}{2}}|\cM |+6mt'n<2\delta^{\frac{1}{2}}|\cM |$ hyperedges in $\mathcal{F}_i$ intersect $H_i$ with $u_1$. This completes the proof.
        \end{proof}
        
        According to (\ref{eq: sum of Fi}) and Claim \ref{claim: bound of Fi}, for every $\mathcal{F}_i$, we have
  \begin{equation*}
            \begin{aligned}
                |\cF_i|&\geq  3t\binom{\lfloor n/3\rfloor-t}{2}+(1-3\delta^{\frac{1}{2}})|\cM |-\binom{(3m+1)t}{2}n\\
                &-\binom{(3m+1)t}{3}-(t-1)\left( 3\frac{(n/3)^2}{2}+(3m+1)\delta^{\frac{1}{2}}|\cM |\right)\\
                &\geq \frac{n^2}{6}-O(n)+\left(1-(t-1)(3m+1)\delta^{1/2}-3\delta^{\frac{1}{2}} \right)|\cM|
                \\
                &> (3m+1)\delta^{1/2}|\cM|.
            \end{aligned}
        \end{equation*}
     The last inequality holds when $\delta^{1/2}<\frac{1}{(t+2)(3m+1)}$ and $n$ is sufficiently large.
       
        Then, Claim \ref{claim: bound of Fi} implies that all but at most $2\delta^{\frac{1}{2}}|\cM |$ hyperedges in $\cF_i$ intersect $H_i$ with one fixed vertex. Let $T$ denote the set of these $t'$ fixed vertices.  Together with the hyperedges in $\cE_1 $, $\cE_2^1 $ and $\cE_2^2 $ that do not intersect with $S_0$, we have proved that all but at most $4t\delta^{\frac{1}{2}}|\cM |$ hyperedges in $\cE $ intersect with these $t$ fixed vertices. Now we have $t'=t$, because the number of hyperedges in $\cE $ containing one vertex is at most $3\binom{n/3}{2}$.

         When $\cM =\emptyset$, all the hyperedges in $\cE $ are based.
        By a similar but simpler analysis, we can also prove that each hyperedge in $\cE $ intersects with a set $T$ of $t$ fixed vertices, we omit the details here.
        Then it remains to maximize the number of hyperedges in $\cE $ intersecting with $T$.

        Note that the number of hyperedges intersecting with $T$ is at most ${n \choose 3}-{n-t \choose 3}$. Let $\cK=\cK(A,B,C)$ denote the complete $3$-partite $3$-uniform hypergraph with parts $A,B,C$.
        Let $\cK(T)=\{E=abc\in E(\cK):E\cap T\neq \emptyset\}$.
        The number of hyperedges in $\cG$ is at most $e(\cK)+{n \choose 3}-{n-t \choose 3}-|\cK(T)|$.
        We let $|A|\geq |B|\geq |C|$ and let $T_A$, $T_B$, $T_C$ be the sets of vertices in $T$ contained in $A,B,C$ respectively. 
        Then we have $|T_A|\geq |T_B|\geq |T_C|$. Otherwise, suppose $|T_A|<|T_C|$, and let $T'$ be the set obtained from $T$ by moving one vertex in $T_C$ to $T_A$. It is easy to check that the number of hyperedges in $\cK(T)$ will not increase.
        Then if $|A|-|C|\geq 2$, we can move one vertex in $A$ to $C$, and it is easy to check that the number of hyperedges in $\cK(A,B,C)$ will increase, while the number of hyperedges in $\cK(T)$ will not increase, a contradiction to the extremality of $\cG$. Thus, we have $|A|-|C|\leq 1$, i.e.,  $\cK$ is balanced. If $|T_A|-|T_C|\geq 2$, then let $T'$ be the set obtained from $T$ by moving one vertex in $T_A$ to $T_C$. It is easy to check the number of hyperedges in $\cK(T')$ is less than $\cK(T)$, a contradiction. Thus, we have $|T_A|-|T_C|\leq 1$. This completes the proof.
\end{proof}

\bigskip
\textbf{Funding}: 
The research of Zhao is supported by the China Scholarship Council (No. 202506210250) and
the National Natural Science Foundation of China (Grant 12571372).

The research of Xin is supported by the National Natural Science Foundation of China (Nos. 12131013 and 12471334), Shaanxi Fundamental Science Research Project for Mathematics and Physics (No. 22JSZ009) and the China Scholarship Council (No. 202406290241). 

The research of Gerbner is supported by the National Research, Development and Innovation Office - NKFIH under the grant KKP-133819.

The research of Miao is supported by the China Scholarship Council (No. 202406770056).

The research of Wang is supported by the China Scholarship Council (No. 202506210200) and
the National Natural Science Foundation of China (Grant 12571372). 

The research of Zhou is supported by the National Natural Science Foundation of China (Nos. 12271337 and 12371347) and the China Scholarship Council (No. 202406890088).

\end{document}